\documentclass{article}

\usepackage[T1]{fontenc}
\usepackage{amsmath, amsthm}
\usepackage{amssymb}
\usepackage{amscd}
\usepackage{graphicx}
\usepackage{enumitem}
\usepackage{float}
\usepackage{subfig}
\usepackage[all]{xy}

\newtheorem{Theorem}{Theorem}[section]
\newtheorem{Corollary}[Theorem]{Corollary}
\newtheorem{Lemma}[Theorem]{Lemma}

\newtheorem{Proposition}[Theorem]{Proposition}

\newtheorem{Example}[Theorem]{Example}

\def\set-up{\noindent {\em Set-Up:}\ }

\def\<{\langle}
\def\>{\rangle}

\def\'{\,^\prime}
\def\"{\,^{\prime\prime}}

\begin{document}

\title{On $\textbf{u}$-substitutions for group presentations}
\author{Kirk McDermott}
\maketitle

\begin{abstract} 

\noindent We investigate groups $\hat{G}$ given by a presentation $\mathcal{P}=\langle \textbf{x}: \textbf{r} \rangle$ whose relators $\textbf{r} \subseteq F(\textbf{x})$ are comprised of a set of subwords in $F(\textbf{x})$, i.e. $\textbf{r}$ admits a $\textbf{u}$-substitution in the sense that there exists a homomorphsim $\epsilon: F(\textbf{u}) \rightarrow F(\textbf{x})$ and a subset $\mathbf{v} \subseteq F(\textbf{u})$ such that $\textbf{r}=\epsilon(\textbf{v})$. Equivalently, $\mathcal{P}= \langle \textbf{x}: \epsilon (\textbf{v}) \rangle$ is referred to as the composition of the presentation $\mathcal{G}= \langle \textbf{u}: \textbf{v} \rangle$ with $\mathcal{H}= \langle \textbf{x}: \epsilon(\textbf{u}) \rangle$ of the groups $G$ and $H$, respectively. We survey known results and record structural properties which do not explicitly appear in the literature, e.g. that there is the relative presentation $\langle G, \textbf{x}: \textbf{u}= \epsilon(\textbf{u}) \rangle$ for $\hat{G}$. Thus there is a natural map $\epsilon: G \rightarrow \hat{G}$ and we may consider the associated problems (e.g. injectivity, finiteness). As an application, we investigate the class of groups $\mathcal{G}(\mathcal{B})$ obtained from substituting a presentation of the trivial group into a deficiency one group presentation for $\mathbb{Z}$. Our results show $\mathcal{G}(\mathcal{B})$ is a proper subset of the class of 2-knot groups and properly contains all classical knot groups. A subfamily is investigated which includes the group of the trefoil knot and uses Higman's presentations for the trivial group.

\end{abstract}

\section{Introduction}

For $n>0$ and $\mathbf{x}= \{ x_0, x_1, \dots, x_{n-1} \}$, let $F(\mathbf{x})$ denote the free group with generating set $\mathbf{x}$. A word $w \in F(\mathbf{x})$ is called composite in $\textbf{s} \subseteq F(\textbf{x})$ if it is comprised of subwords $\textbf{s}=\{w_0, w_1, \dots, w_{k-1} \}$ for some $k>0$. Let $\textbf{u}=\{u_0, u_1, \dots, u_{k-1} \}$ and for $F(\textbf{u})$ the free group on $\textbf{u}$, define the homomorphism $\varepsilon: F(\mathbf{u}) \rightarrow F(\mathbf{x})$ by $\varepsilon(u_i)=w_i$ for $0 \leq i< k$. Thus $w$ is in the image of $\epsilon$ and there exists $v\in F(\textbf{u})$ such that $\epsilon(v)=u$. In this situation, we say $w$ admits a $\textbf{u}$-substitution in the substitution set $\epsilon(\textbf{u})=\{w_0, w_1, \dots, w_{k-1} \}$ and call $\epsilon$ a substitution homomorphism. A group presentation $\mathcal{P}=\langle \mathbf{x}: \mathbf{r} \rangle$ with $\mathbf{r}\subseteq F(\mathbf{x})$ is called composite if there exists a pair $(\mathcal{G}, \epsilon)$ where $\mathcal{G}=\langle \textbf{u}: \textbf{v} \rangle$ and $\epsilon: F(\textbf{u})\rightarrow F(\textbf{x})$ is such that $\epsilon(\textbf{v})=\textbf{r}$. Equivalently, given arbitrary $\mathcal{G}=\langle \textbf{u}: \textbf{v} \rangle$ and $\mathcal{H}= \langle \textbf{x}: \textbf{s} \rangle$ with $\textbf{s} \subseteq F(\textbf{x})$ such that $|\textbf{s}|=|\textbf{u}|$, presentations for the groups $G$ and $H$, respectively, then the composition $\mathcal{P}$ of $\mathcal{G}$ with $\mathcal{H}$ is defined in terms of the substitution $\epsilon: F(\textbf{u}) \rightarrow F(\textbf{x})$ with $\epsilon(\textbf{u})=\textbf{s}$ where as before $\mathcal{P}= \langle \textbf{x}: \epsilon(\textbf{v}) \rangle$. Our goal in this paper is to investigate the group $\hat{G}$ given by a composite presentation $\mathcal{P}$, regarded as a construction having input an arbitrary pair $(\mathcal{G}, \epsilon)$ (equivalently $(\mathcal{G}, \mathcal{H} )$). 

Composite words and presentations have a rich history. Indeed, Nielsen transformations are regarded as substitutions, composite one relator presentations were studied by Magnus \cite[Theorem N5]{MagKarSol}, and Fox \cite{FoxFDI} discovered a chain rule for composite words in his free differential calculus. At the same time, there does not appear to be a systematic treatment of composite group presentations in the literature \footnote{ The same conclusion was reached in \cite[p. 188]{MyaAC}.}. We hope to fill this void in the present paper by developing some of the theory and relating composite presentations to some well studied problems in group theory, where throughout this paper we will survey the situation as best as possible, too. Our starting point is the following which may be well known but does not explicitly appear elsewhere.

\begin{Theorem} \label{Theorem: intro short exact} Let $\epsilon: F(\mathbf{u}) \rightarrow F(\mathbf{x})$ be a homomorphism and $G \cong \langle \mathbf{u}: \mathbf{v} \rangle$ with $\mathbf{v} \subseteq F(\mathbf{u})$. Then for $\hat{G}$ the group given by the composite presentation $\mathcal{P}= \langle \mathbf{x}: \epsilon(\mathbf{v}) \rangle$, the substitution homomorphism induces a map $\varepsilon:G \rightarrow \hat{G}$ and there is an exact sequence of homomorphisms
\begin{align} \label{Eqn: Short exact intro} 
1 \rightarrow \langle \langle \epsilon(G) \rangle \rangle \xrightarrow{i} \hat{G} \xrightarrow{\pi} H \rightarrow 1
\end{align}
where $H\cong \langle \mathbf{x}: \epsilon(\mathbf{u}) \rangle$ and $\langle \langle \epsilon(G) \rangle \rangle$ is the normal closure of $\epsilon(G)$ in $\hat{G}$.
\end{Theorem}

The proof of the theorem and other preliminary results appear in Section \ref{Section: background}. In short, the result follows from adjoining the generators $\textbf{u}$ to $\mathcal{P}$ with defining relations $\textbf{u}=\epsilon(\textbf{u})$, and then rewriting $\textbf{r}$ in terms of $\textbf{u}$ to obtain $\textbf{v}$. This also provides the relative presentation
\begin{align} \label{Eqn: rel intro}
\hat{\mathcal{G}} = \langle G, \mathbf{x}: \epsilon(u_i)u_i^{-1}, \ 0\leq i< k\rangle
\end{align}
for $\hat{G}$ where the substitution homomorphism is the natural map $\epsilon: G \rightarrow \hat{G}$. Among other things, this allows us to frame injectivity of the induced map in terms of the adjunction problem from equations over groups (see Section \ref{Section: eqn groups}). 

An analysis of the normal closure of $\epsilon(G)$ in $\hat{G}$ appears in Section \ref{Section: The normal closure G in hat G}. The difference between $\epsilon(G)$ and its normal closure in $\hat{G}$ is hard to describe, enough so that the group $\hat{G}$ has interesting triviality and finiteness problems. Indeed, if we ask whether $\hat{G}$ is nontrivial whenever $G$ is nontrivial, then under certain circumstances we are considering the Kervaire conjecture and related problems (e.g. for $k>1$ adjunctions). Experimentation in GAP \cite{GAP} suggests that even if $H$ is trivial and $G$ is finite cyclic, then the group $\hat{G}$ is typically infinite. A small number of finite and nontrivial groups $\hat{G}$ were found using GAP and these are recorded in Subsection \ref{Subsection: finiteness} on the finiteness problem. We also find endomorphisms $\Phi_{\epsilon} \in \mathrm{End}(F(\textbf{x}))$ where $\hat{G}_a \cong \langle \textbf{x}: \Phi_{\epsilon}^a (\textbf{x}) \rangle$ is finite for all $a>0$. Incidentally, such a substitution $\Phi_{\epsilon}$ is an endomorphism of $F(\textbf{x})$ which fails to be surjective, but has the property that the subgroup normally generated by $\Phi_{\epsilon}^a(\textbf{x})$ has finite index in $F(\textbf{x})$ for all $a>0$.

We define the cellular model $L$ for the relative presentation $\hat{\mathcal{G}}$ in Subsection \ref{Subsection top and asph}. This puts us in position to use the asphericity hypothesis for a relative presentation (e.g. \cite{BogWilEdjAsph}). Asphericity of $L$ implies, for example, that $G$ embeds in $\hat{G}$. In forthcoming work the author investigates an essential map $S^2 \rightarrow L$ that arises when a substitution is not aspherical, i.e. a representative of a nontrivial element in the homotopy group $\pi_2(L,K)$ for $K$ a $K(G,1)$.

As as application, in Section \ref{Subsection: normal closure} we introduce the class of groups $\mathcal{G}(\mathcal{B})$ defined in terms of those composite presentations $\mathcal{B}$ obtained from composing a deficiency one presentation $\mathcal{G}$ for $G=\mathbb{Z}$ with a presentation $\mathcal{H}$ of the trivial group $H=1$. The resulting finitely presented groups have deficiency one, weight one (are normally generated by one element), and have infinite cyclic abelianization. Thus $\mathcal{G}(\mathcal{B})$ are the groups of 2-knots \cite{KerCond}. It is known that there exist 2-knot groups of deficiency zero, hence $\mathcal{G}(\mathcal{B})$ is a proper subset of the 2-knot groups. On the other hand, it is easy to see that all deficiency one Wirtinger presentations admit such a substitution, so $\mathcal{G}(\mathcal{B})$ contains all (classical) 1-knot groups. A specific family $Q(k,l)$ of groups related to the class $\mathcal{G}(\mathcal{B})$ are investigated in Section \ref{Section: QKL} using Higman's presentations $\mathcal{Q}$ for the trivial group \cite{HigInf}. We show this family contains the knot of the trefoil, in addition to 2-knot groups which are not 1-knot groups. Hence $\mathcal{G}(\mathcal{B})$ is properly contained between the classes of $1$-knot groups and $2$-knot groups.

Cyclically presented groups demonstrate the situation in an efficient way and we will use them extensively in our applications (see Subsection \ref{Subsection: cyc pres}). In this context, the idea of studying $\hat{G}=G_n(v \circ u)$ via $G=G_n(v)$ is introduced in \cite{BogShi} and arose in classifying finiteness and shift dynamics for positive words of length four in \cite{BogParFour}. To the best of the author's knowledge, \cite{BogShi} is the first to consider the embedding $G\rightarrow \hat{G}$ using the theory of equations over groups. The examples in Section \ref{Section: QKL} and some of the general results were originally stated in the special case of cyclically presented groups in the author's dissertation \cite{McDis}, and he would like to thank his adviser W. A. Bogley for his help then and now.

The paper is arranged as follows. Section \ref{Section: background} begins with some additional background and references. This section also contains preliminaries including the relative presentation $\hat{\mathcal{G}}$ for $\hat{G}$, the cellular models for the group presentations and asphericity, and definitions and results concerning cyclically presented groups. Section \ref{Section: The normal closure G in hat G} is on the normal closure of $\epsilon(G)$ in $\hat{G}$. Here we also consider the finiteness problem for $\hat{G}$, introduce the class of groups $\mathcal{G}(\mathcal{B})$ of weight one, and also consider the iterated compositions $\hat{G}_a \cong \langle \textbf{x}: \Phi^a (\textbf{x}) \rangle$. Section \ref{Section: eqn groups} focuses on the substitution homomorphism $\epsilon$ and contains the applications to equations over groups. The surjectivity problem is also considered and this has applications to presentations of the trivial group, too. Finally, the family of groups $Q(k,l)$ are defined and investigated in Section \ref{Section: QKL}; we show there is also a subfamily of infinite groups generated by involutions.

\section{Background and preliminary results} \label{Section: background}

We begin with some additional background and known results. When there are $k=n$ substitutions and the free groups $F(\mathbf{u})$ and $F(\mathbf{x})$ have the same rank, the substitution homomorphism may be regarded as an endomorphism $\Phi_{\epsilon}$ of $F(\mathbf{x})$, where $\Phi_\epsilon=\epsilon \circ \mathrm{Id}$ and $\mathrm{Id}:F(\mathbf{x})\rightarrow F(\mathbf{u})$ is the identity map given by $x_i \mapsto u_i$ for $0\leq i<n$. Here we regard $\mathbf{v}$ as in the alphabet $\mathbf{x}$ and present $G$ by $\mathcal{G}= \langle \textbf{x}: \textbf{v} \rangle$. It has been known for some time that substitutions $\Phi_{\epsilon}$ correspond to endomorphisms of $F(\mathbf{x})$, where in the study of Nielsen transformations it is a classical result \cite[Lemma 3.3]{MagKarSol} that a substitution $\Phi_{\epsilon}$ which is an automorphism of $F(\mathbf{x})$ induces an automorphism of an arbitrary group $G=\hat{G}$ given by $\mathcal{G}$ if an only if 
\begin{align}
&\Phi_{\epsilon} (v) = 1_G, \ \mathrm{for \ all} \ v\in \mathbf{v}, \ \mathrm{and} \label{Condition 1} \\
&\Phi_{\epsilon}^{-1}(v) = 1_G, \ \mathrm{for \ all} \ v \in \mathbf{v}, \label{Condition 2}
\end{align}
where $\Phi_{\epsilon}^{-1}$ is the substitution $x_i \mapsto g_i$ with $g_i \in F(\mathbf{x})$ such that $\Phi_{\epsilon} (g_i)=x_i$ for $0 \leq i < n$. Condition (\ref{Condition 1}) provides that $\Phi_{\epsilon}$ induces an endomorphism of $G$, and condition (\ref{Condition 2}) corresponds to injectivity (the induced map of a free group automorphism is always surjective). For an arbitrary pair $(\mathcal{G}, \Phi_{\epsilon})$, conditions (\ref{Condition 1}) and (\ref{Condition 2}) will hold true in $G$ just as often as the relations $\textbf{v}$ hold in $\hat{G}$, which is to say very rarely. Thus in forming $\hat{G}$ and relaxing the requirement that $\Phi_{\epsilon}$ be an automorphism of $F(\mathbf{x})$, we have put ourselves in a position where given the pair $(\mathcal{G}, \Phi_{\epsilon})$ there is always the group $\hat{G}$ with induced homomorphism $\Phi_{\epsilon}: G \rightarrow \hat{G}$.

The idea of composing two presentations $\mathcal{G}, \mathcal{H}$ of the trivial group is well-known, for then $\hat{G}$ must be trivial, too. The earliest such reference to the author's knowledge is \cite{NeuPro}, where B. H. Neumann observes that the presentation $\mathcal{Q}_n$ obtained by successively composing Higman's presentation $\mathcal{Q}=\langle x_0, x_1, x_2: x_ix_{i+1}x_i^{-1}x_{i+1}^{-2}, \ i=0,1,2 \rangle$ of the trivial group \cite{HigInf} with itself $n$-times produces complex group presentations which can be used to test a computer's ability to perform coset enumeration. Then \cite{BurBal} is the first of several papers (e.g. \cite{HavTri}, \cite{MyaAC}) that discuss composing balanced presentations of the trivial group as it relates to the Andrews- Curtis conjecture. Balanced presentations of the trivial group have well known applications to topology, and there the presentations $\mathcal{Q}_n$ have been studied, too. In \cite{HowAsphRib}, Howie proves $\mathcal{Q}_n$ does not yield a counterexample to Whitehead's asphericity conjecture using the method of deleting a relator from a balanced presentation of the trivial group. 

There are a considerable number of results in one-relator group theory using composite presentations. We do not attempt to summarize them all here, nor do we make a serious attempt to consider the usual problems. For a classical result, see Meskin \cite{MesRF}  who considers residual finiteness for relators of the form $u^{-1}v^luv^m$ for $u,v \in F(\textbf{x})$; more recently, in \cite{BaumMillRF} it is shown that most one relator groups with the form $u^{u^v}=u^2$ are not residually finite. See also Pride \cite{PriRF} who considers the problem for composite presentations with a certain unique max-min property. Small cancellation theory may be used to investigate $\hat{G}$ if $\epsilon(u_i)$ is short relative to the word length of the relations in $ \epsilon(\mathbf{v})$. This was done in \cite{JuhExt} and \cite{JuhSpell} for composite presentations satisfying certain combinatorial and small cancellation conditions.

Another classical result from one relator group theory states that if for $u,  \bar{u} \in F(\mathbf{x})$ the groups given by $\langle \mathbf{x}: u^a \rangle$ and $\langle \mathbf{x}: \bar{u}^a \rangle$ are isomorphic, then so are the groups given by $\langle \mathbf{x}: u \rangle $ and $\langle \mathbf{x}: \bar{u} \rangle$ \cite[Corollary 4.13.1]{MagKarSol}. It is interesting to frame this result as follows. The presentation $\mathcal{G}= \langle \mathbf{u}: u_0^a \rangle$ has the property that $\hat{G}$ determines $H$. Compare with \cite[Section 3.7]{MagKarSol} and the remark following Theorem N5. From this perspective, our examples of finite groups in Section \ref{Subsection: finiteness} can be used to show that $\hat{G}$ is not an invariant of $(\mathcal{G}, H)$ or $(G, \epsilon)$ in general.


\subsection{Preliminaries and the relative presentation $\hat{\mathcal{G}}$} \label{Subs: relative}


For $n>0$ and $\mathbf{x}= \{x_0, x_1, \dots, x_{n-1} \}$, let $\mathbf{r} \subseteq F(\mathbf{x})$ and $\mathcal{P} = \langle \mathbf{x}: \mathbf{r} \rangle$ be a composite presentation for the group $\hat{G}$. Thus there exists a substitution homomorphism $\epsilon: F(\mathbf{u}) \rightarrow F(\mathbf{x})$ with $k=|\textbf{u}|>0$ and a group presentation $\mathcal{G}= \langle \mathbf{u}: \mathbf{v} \rangle$ such that $\mathbf{r}=\epsilon(\mathbf{v})$. Our definition implies every group presentation is composite using $k=n$ substitutions and the identity homomorphism $u_i \mapsto x_i$ for $0 \leq i<n$. To distinguish from this and other permutations of the generating sets, we say a word $w \in F(\mathbf{x})$ admits a nontrivial $\textbf{u}$-substitution if $w = \epsilon(v)$ for some $v \in F(\mathbf{u})$ is such that the image $\epsilon(u_j)$ of at least one generator $u_j \in \textbf{u}$ appearing in the word $v$ has freely reduced word length greater than one.

\begin{Example}
\begin{itemize}

\item[ (i) ] \label{Example: empty word composite} Define $\epsilon$ by $u_i \mapsto x_i x_{i+1}^{-1}$ for $0 \leq i <k= n$, with subscripts taken modulo $n$, and $v=u_0 u_1 \dots u_{n-1}$. Then $\epsilon(v) =1 \in F(\mathbf{x})$ and so the empty word admits a nontrivial $\textbf{u}$-substitution.

\item[ (ii) ] A word $w\in F(\mathbf{x})$ that is a proper power has the decomposition $w=u^a$ for $a>1$, where $\epsilon$ defined by $u_0 \mapsto u \in F(\mathbf{x})$ and $v=u_0^a \in F(u)$ satisfies $\epsilon(v)= w$.

\item[ (iii) ] If $v\in F(\textbf{u})$ and $\epsilon(\textbf{u}) \subseteq F(\textbf{x})$ are all words with positive exponents, then their reduced length satisfies $|\epsilon(v)|=|u|_{\mathbf{x}} |v|_{\mathbf{u}}$.

\item[ (iv) ] If $\Phi \in \mathrm{End}(F(\textbf{x}))$ is periodic, i.e. there exists $k>0$ such that $\Phi^k(x_i)=x_i$ for all $x_i \in \textbf{x}$, then $\Phi(\textbf{v})=\textbf{v}$ and $G \cong \hat{G}_k = \langle \textbf{x}: \Phi^k(\textbf{v}) \rangle$.

\end{itemize}
\end{Example}

Variations of Example \ref{Example: empty word composite} $(i)$ arise from the fact that $\epsilon(v) \in F(\mathbf{x})$ is a decomposition of the empty word if and only if $v$ is in the kernel of a substitution $\epsilon: F(\textbf{u}) \rightarrow F(\textbf{x})$. Example \ref{Example: empty word composite} (ii) in the context of one-relator groups is considered in Example \ref{Example: one relator}.

We now show how the group $\hat{G}$ may be obtained from $G$ by adjoining generators and defining relations. The resulting relative presentation $\hat{\mathcal{G}}$ for $\hat{G}$ has many consequences. Theorem \ref{Theorem: intro short exact} follows immediately, and it will allow us to frame injectivity of $\epsilon: G \rightarrow \hat{G}$ in terms of the adjunction problem from equations over groups. Furthermore, the topological analogue of obtaining $\hat{G}$ from $G$, in terms of cellular models and the asphericity hypothesis, is discussed in Subsection \ref{Subsection top and asph}.

\begin{Lemma} \label{Lemma: relative hat G} Suppose $\hat{G}$ is given by the group presentation $\mathcal{P}= \langle \mathbf{x}: \mathbf{r} \rangle$ with $\mathbf{r}\subseteq F(\mathbf{x})$ and there exists $\mathcal{G}= \langle \mathbf{u}: \mathbf{v} \rangle$ and $\epsilon: F(\mathbf{u}) \rightarrow F(\mathbf{x})$ such that $\epsilon(\mathbf{v})= \mathbf{r}$. Then the group $\hat{G}$ is given by the relative presentation
\begin{align*} 
\hat{\mathcal{G}} = \langle G, \mathbf{x}: \epsilon(u_i)u_i^{-1}, \ u_i \in \textbf{u} \rangle
\end{align*}
where the natural map $G \rightarrow \hat{G}$ is induced by the substitution homomorphism $\epsilon$.
\end{Lemma}

\begin{proof} Let $m=|\textbf{r}|$ be the number of relators in $\mathcal{P}$ and $k=|\textbf{u}|$ the number of substitutions. Using Tietze transformations to adjoin the generators $\textbf{u}$ with defining relations $u_i=\epsilon(u_i)$ for $0 \leq i <k$, and then rewriting $r_j$ in the substitution set $\epsilon(\mathbf{u})$ to get $v_j$ for $0 \leq j< m$, via multiplication by the relators $\epsilon(\textbf{u})\textbf{u}^{-1}$ and inversion and cylic permutation of the $\textbf{r}$, we have
\begin{align*}
\mathcal{P} &= \langle x_0, \dots , x_{n-1} : r_0, \dots, r_{m-1} \rangle \\
& \cong \langle x_0, \dots, x_{n-1}, u_0, \dots, u_{k-1}: r_0, \dots, r_{m-1}, \epsilon(u_i)u_i^{-1}, \ 0 \leq i< k \rangle \\
&\cong \langle x_0, \dots, x_{n-1}, u_0, \dots, u_{k-1}: v_0, \dots, v_{m-1},\epsilon(u_i)u_i^{-1}, \ 0 \leq i< k \rangle
\end{align*}
from which the result follows. \qed
\end{proof}

The substitution homomorphism may map a proper subset of the generators $\mathbf{u}$ of $F(\mathbf{u})$ onto $\mathbf{x}$ and still yield useful information about $\hat{G}$. Call a presentation $\mathcal{P}$ partially composite if it is composite in $(\mathcal{G}, \epsilon )$ with $\epsilon(u_i) \in \mathbf{x}$ for some but not all $u_i \in \textbf{u}$. For a partially composite presentation, in terms of the relative presentation $\hat{\mathcal{G}}$ of $\hat{G}$ from Lemma \ref{Lemma: relative hat G}, any adjoined relation of the form $\epsilon(u_i)=x_j$, along with the corresponding generator $x_ j \in \mathbf{x}$, may be deleted from $\hat{\mathcal{G}}$ upon replacing any instances of $x_j$ with $u_i$ in the adjoined relations. This reduces $\hat{\mathcal{G}}$ as we show now. Compare with Example \ref{Example: not injective E} and Lemma \ref{Lemma: E relative}.

\begin{Example} \label{Example: partial comp} Let $\mathcal{P}=\langle a, x: a^3, (xaxax^{-1}a^{-2})^2 \rangle$ where $\mathbf{x}=\{a, x \}$. First consider the case where we regard $\mathcal{P}$ as a composite presentation with $k=2$ substitutions. For $G$ the group presented by $\mathcal{G}= \langle A, U: A^3, U^2 \rangle$ where $\mathbf{u}=\{A, U \}$ and $\epsilon: F(\mathbf{u}) \rightarrow F(\mathbf{x})$ is defined by $A \mapsto a$ and $U \mapsto xaxax^{-1}a^{-2}$, we have 
$$\hat{\mathcal{G}}= \langle G, a, x: aA^{-1}, xaxax^{-1}a^{-2}U^{-1} \rangle .$$
Now $\mathcal{P}$ is also partially composite and observe the adjoined generator $a$ and relator $aA^{-1}$ may be deleted so that $\hat{G} \cong \langle G, x: xAxAx^{-1}A^{-2}U^{-1} \rangle$.
\end{Example}


\subsection{Cellular models and asphericity} \label{Subsection top and asph}

Let $\mathcal{P}=\langle \mathbf{x}: \mathbf{r} \rangle$ where $\mathbf{x}=\{ x_0, \dots, x_{n-1} \}$ and $\mathbf{r}=\{ r_0, \dots, r_{m-1} \} \subseteq F(\mathbf{x})$ for $n,m>0$, and $G(\mathcal{P})$ the group presented by $\mathcal{P}$. The 2-dimensional cellular model $K=K_{\mathcal{P}}$ of the group presentation $\mathcal{P}$ has 1-skeleton $K^{(1)}$ a wedge of $n$ circles, where each circle has minimal CW-complex structure $S^1_x=c^0 \cup c^1_x$ and is indexed by the generators $x \in \mathbf{x}$. Orient each loop $S^1_x$ with counterclockwise positive, and attach 2-cells $c_{r}^{2}$ to $K^{(1)}$, for each $r \in \mathbf{r}$, via the attaching map $\dot{\phi}_{r}:S^{1}\rightarrow K^{(1)}$ which spells out the word $r$ according to the labels given to $K^{(1)}$. The resulting CW-complex $K$ has fundamental group isomorphic to $G(\mathcal{P})$ and is called the cellular model of $\mathcal{P}$. Furthermore, the presentation $\mathcal{P}$ is called aspherical if the second homotopy group $\pi_2(K)$ is trivial.

Next we define the cellular model and asphericity for a relative presentation as in \cite{BogWilEdjAsph}. For $n,m>0$, let $H$ be a group, $\mathbf{y}=\{y_0, \dots, y_{n-1} \}$ a set disjoint from $H$, and $\mathbf{s}=\{s_0, \dots, s_{m-1} \}$ a set of cyclically reduced elements in $H \ast F(\mathbf{y})$.  The relative presentation $\mathcal{Q}=\langle H, \mathbf{y}: \mathbf{s} \rangle$ determines the group $Q=H \ast F(\mathbf{y}) / N$, for $N$ the normal closure of $\mathbf{s}$ in $H\ast F(\mathbf{y})$. Let $K$ be an Eilenberg-MacLane space of type $K(H,1)$. Form
$$L=K \vee \underset{y\in\mathbf{y}}{\bigvee} S^1_y \cup \underset{s\in\mathbf{s}}{\bigcup} c_s^2,$$
where the 2-cells $c_{s}^2$ are attached via $\dot{\phi}_{s}:S^{1}\rightarrow L^{(1)}$ spelling out $s$ in terms of the labels along the (oriented) wedge of circles. The resulting space $L$ is defined to be the cellular model of $\mathcal{Q}$, and the relative presentation $\mathcal{Q}$ is called aspherical if the second homotopy group $\pi_2(L,K)$ is trivial. Note that the definition of asphericity of $\mathcal{Q}$ is equivalent to assuming the natural map $H \rightarrow G$ is injective and that $K$ is a $K(Q,1)$ complex \cite[Lemma 2.3]{BogWilEdjAsph}.

Returning to our situation, the group $\hat{G}$ is given by the composite presentation $\mathcal{P}$, and also the relative presentation $\hat{\mathcal{G}}$ from Lemma \ref{Lemma: relative hat G}. We have the following concerning asphericity.

\begin{Theorem} \label{Theorem: asph rel asph asph} The 2-dimensional cellular model of the composite group presentation $\mathcal{P}=\langle \mathbf{x}: \mathbf{r} \rangle$ of $\hat{G}$ is homotopy equivalent to a space obtained from the cellular model of $\mathcal{G}=\langle \mathbf{u}: \mathbf{v} \rangle$ by adjoining cells in dimensions 1 and 2. If both the presentation $\mathcal{G}$ of $G$ and the relative presentation $\hat{\mathcal{G}}$ of $\hat{G}$ are aspherical, then $\mathcal{P}$ is aspherical.
\end{Theorem}

\begin{proof} Observe $\mathcal{G}$ is a subpresentation of $\tilde{\mathcal{P}}=\langle \mathbf{u}, \mathbf{x}: \mathbf{v}, \epsilon(\textbf{u})\textbf{u}^{-1} \rangle$, which can be obtained from $\mathcal{P}$ via Tietze transformations as in the proof of Lemma \ref{Lemma: relative hat G}. Thus the model for $\mathcal{P}$ is homotopy equivalent to a presentation complex having subspace a model for $\mathcal{G}$. Suppose $\mathcal{G}$ is aspherical, i.e. its model $K=K_{\mathcal{G}}$ has trivial $\pi_2(K)$. Then $K$ is a $K(G,1)$ and so it may be used to construct the model $L$ for the relative presentation $\hat{\mathcal{G}}$. Thus $L$ is identically the 2-dimensional cellular model of $\tilde{\mathcal{P}}$. Now $\pi_2(L,K)$ is also trivial since we assume that $\hat{\mathcal{G}}$ is aspherical, and so by the long exact sequence in homotopy
$$\dots \rightarrow \pi_2(K) \rightarrow \pi_2(L) \rightarrow \pi_2(L,K)\rightarrow \pi_1(K) \rightarrow \dots$$
for the subspace $K \subseteq L$ we have that $\pi_2(L)$ is trivial. Thus $\tilde{\mathcal{P}}$ is aspherical and so is the homotopy equivalent model for $\mathcal{P}$, too.
\qed
\end{proof}

The proof of Theorem \ref{Theorem: asph rel asph asph} contains the presentation $\tilde{\mathcal{P}}$ for $\hat{G}$- this is the ordinary (or lifted) presentation associated with the relative presentation $\hat{\mathcal{G}}$. Let $M$ denote the cellular model for $\tilde{\mathcal{P}}$ where the cellular model $K$ for $\mathcal{G}$ is a subspace of $M$. This topological framework with subspace pair $(M,K)$ is not explicitly stated in the literature. As an application, examples for testing Whitehead's asphericity conjecture maybe be constructed using composite presentations where $\tilde{\mathcal{P}}$ is apsherical. Indeed, it is possible that $\mathcal{G}$ is not aspherical. For comparison, recall examples for testing this conjecture may also be constructed by deleting a relator from a balanced presentation of the trivial group, as was done with compositions of Higman's presentation for the trivial group in \cite{HowAsphRib}.

The following are well-known results about aspherical relative presentations tailored to our situation. For details and additional implications (e.g. cohomology results) see \cite[Theorem 2.4]{BogWilEdjAsph}.

\begin{Theorem} \label{Theorem: relative aspherical} Assume the relative presentation $\hat{\mathcal{G}}= \langle G, \mathbf{x}: \epsilon(u_i)u_i^{-1}, \ 0\leq i< k\rangle$ of $\hat{G}$ is aspherical. Then 
\begin{itemize}
\item[(i)] the substitution homomorphism $\epsilon: G \rightarrow \hat{G}$ is injective,
\item[(ii)] if $G$ is nontrivial, then each finite subgroup of $\hat{G}$ is conjugate in $\hat{G}$ to a subgroup of $\epsilon(G)$, and
\item[(iii)] if both $G$ and $H\cong \langle \mathbf{x}: \epsilon(\textbf{u}) \rangle$ are nontrivial, then $\hat{G}$ is infinite.
\end{itemize} 
\end{Theorem}

\begin{proof} Relative asphericity implies $G$ injects into $\hat{G}$. This follows from $\pi_2(L,K)$ being trivial in the long exact homotopy sequence in the proof of the previous theorem. Part [ii] follows from \cite[Theorem 2.4 (c)]{BogWilEdjAsph} which uses a group cohomology theorem due to Serre \cite[p. 139]{SerreThm}. For part [iii], suppose $G$ and $H$ are nontrivial. Theorem \ref{Theorem: intro short exact} provides a quotient map $\pi:\hat{G} \twoheadrightarrow H$ with kernel $ \langle \langle \epsilon(G) \rangle \rangle_{\hat{G}}$, so there exists a nontrivial element $g\in \hat{G}$ such that $q(g)$ is nontrivial in $H$. If $g$ has finite order, then the subgroup generated by $g$ is conjugate to a subgroup of $\epsilon(G)$ by part [ii]. But this implies $g \in  \langle \langle \epsilon_u(G) \rangle \rangle_{G_n(v \circ u)}$, in which case $q(g)$ is trivial. Thus $g$ must have infinite order, proving $\hat{G}$ is infinite. \qed
\end{proof}

\subsection{Definitions and results for cyclically presented groups} \label{Subsection: cyc pres}
A group $G$ is called cyclically presented if there exists $n>0$ and $w \in F(\mathbf{x})$ such that $G \cong G_n(w)$ is given by the cyclic presentation $\mathcal{G}_n(w) = \langle \mathbf{x}: \theta^i_\mathbf{x}(w), \ 0 \leq i<n \rangle$, where the shift automorphism $\theta_\mathbf{x} \in \mathrm{Aut}(F(\mathbf{x}))$ (respectively $\theta_\mathbf{u}\in \mathrm{Aut}(F(\mathbf{u}))$) is defined by $\theta_\mathbf{x} (x_i)=x_{i+1}$ for $0\leq i<n$, with subscripts taken modulo $n$. A word $u\in F(x_0, \dots, x_{n-1})$ determines the substitution homomorphism
$$\epsilon_{u}:F(\mathbf{u})\rightarrow F(\mathbf{x})$$
defined by $\epsilon_u(u_{i}) =\theta^{i}_{\mathbf{x}}(u)$ for $0 \leq i< n$, and in this situation we will write $v \circ u = \epsilon_u(v)$. 

The shift automorphism induces an automorphism of a cyclically presented group and to begin we have the following.

\begin{Lemma} \label{Lemma: shift sub commute} Let $n>0, u \in F(\mathbf{x})$, and $v\in F(\mathbf{u})$. Then the substitution homomorphism $\epsilon_u : F(\mathbf{u}) \rightarrow F(\mathbf{x})$ defined by $\epsilon_u(u_{i}) =\theta^{i}_{\mathbf{x}}(u)$ for $0 \leq i< n$ induces a homomorphism 
$$\epsilon_u:G_n(v) \rightarrow G_n(v \circ u)$$
that is shift equivariant, i.e. $\epsilon_u \circ \theta_{\mathbf{u}}=\theta_{\mathbf{x}}\circ \epsilon_u$. 
\end{Lemma}

\begin{proof} For $0\leq i <n$, the substitution homomorphism is shift equivariant since each $u_i \in \mathbf{u}$ satisfies 
$$\theta_{\mathbf{x}}(\epsilon_{u}(u_{i}))=\theta_{\mathbf{x}}(\theta^{i}_{\mathbf{x}}(u))=\theta^{i+1}_{\mathbf{x}}(u)=\epsilon_{u}(u_{i+1})=\epsilon_{u}(\theta_{\mathbf{u}}(u_{i})),$$
where the induced map between groups is well defined since each relator $\theta^i_{\mathbf{u}}(v)$ of $\mathcal{G}_n(v)$ maps to $(\epsilon_u \circ \theta^{i}_{\mathbf{u}})(v) = (\theta^{i}_{\mathbf{x}} \circ \epsilon_{u})(v) = \theta^{i}_{\mathbf{x}}(v \circ u)$, a defining relator of $\mathcal{G}_n(v \circ u)$. \qed
\end{proof}

Given a cyclically presented group $G_n(w)$ with $w \in F(\mathbf{x})$, its shift extension $E_n(w)= G_n(w)\rtimes_{\theta_{\mathbf{x}}} \mathbb{Z}_n$ admits the presentation 
$$\mathcal{E}_n(w)= \langle a, x: a^n, W \rangle ,$$ 
where $W$ is obtained from $w$ by the substitution $x_i=a^ixa^{-i}$ for $0 \leq i < n$. The natural inclusion $i_{a,x}:G_n(w) \rightarrow E_n(w)$ is defined by $i_{a,x}(x_i)=a^ixa^{-i}$ for $0 \leq i < n$ so that $i_{a,x}(w)=W$. When $w= v \circ u$ is composite, the shift extension is  a partially composite presentation as in Example \ref{Example: partial comp}, and so we may reduce the relative presentation $\hat{\mathcal{G}}$ from Lemma \ref{Lemma: relative hat G} as follows. 

\begin{Lemma} \label{Lemma: E relative} Let $n>0$, $u\in F(\mathbf{x})$ and $v \in F(\mathbf{u})$. The substitution homomorphism $\epsilon_u:G_n(v) \rightarrow G_n(v \circ u)$ extends to a homomorphism 
$$\epsilon_u: E_n(v) \rightarrow E_n(v \circ u)$$ 
defined by $\epsilon_u(a)=a$ and $\epsilon_u(U)=i_{a,x}(u)$, and the shift extension $E_n(v \circ u)=G_n(v \circ u) \rtimes_{\theta_{\mathbf{x}}} \mathbb{Z}_n$ admits a relative presentation of the form
\begin{align} \label{Presentation: E relative}
\hat{\mathcal{E}}_n(v \circ u) = \langle E_n(v), x : i_{a,x}(u)U^{-1} \rangle  ,
\end{align}
where $E_n(v)\cong \langle a, U: a^n, i_{a,U}(v) \rangle $.
\end{Lemma}


\begin{proof} The group $E_n(v \circ u)$ is presented by $\mathcal{E}_n(v \circ u) = \langle a, x: a^n, i_{a,x}(v \circ u)  \rangle$ and the inclusion $i_{a, U}:F(\mathbf{u}) \rightarrow F(a, U)$ defined by $i_{a, U}(u_i)=a^iUa^{-i}$ extends $\epsilon_u:F( \mathbf{u}) \rightarrow F(\mathbf{x})$ to 
$$\epsilon_u:\mathbb{Z}_n \ast \langle u \rangle \rightarrow \mathbb{Z}_n \ast \langle x \rangle$$ 
by $\epsilon_u(a)=a$ and $\epsilon_u(U)= i_{a,x}(u)$ for $\mathbb{Z}_n \cong \langle a: a^n \rangle$. One checks $\epsilon_u \circ i_{a, U} = i_{a, x} \circ \epsilon_u$ from which it follows that $i_{a,x}(v \circ u) = \epsilon_u(i_{a,U}(v))$.

%
%

Next adjoin the generator $U$ to $\mathcal{E}_n(v \circ u)$ by the relation $U=i_{a,x}(u)$. Working modulo the relation $i_{a, x}(u)U^{-1}=1$, we have $\epsilon_u(U)=i_{a,x}(u) \equiv U$, and since $\epsilon_u(a)=a$,  this implies $\epsilon_u(i_{a, U}(v)) \equiv i_{a, U}(v)$. It now follows that $E_n(v \circ u)$ admits the relative presentation (\ref{Presentation: E relative}). \qed
\end{proof}

For $\mathbb{Z}_n \cong \langle a: a^n \rangle$ and $w\in F(\mathbf{x})$, the canonical retraction
$$\nu=\nu_w: E_n(w) \rightarrow \mathbb{Z}_n$$
is defined by $\nu(a)=a$ and $\nu(x)=1$ and it has the property that $\mathrm{ker}(\nu)\cong G_n(w)$.

\begin{Lemma} \label{Lemma: epsilon extension commutes} Let $n>0$. For $u\in F(\mathbf{x})$ and $v \in F(\mathbf{u})$, let $\nu_{v}$ and $\nu_{v \circ u}$ be the canonical retractions from $E_n(v)$ and $E_n(v \circ u)$ onto $\mathbb{Z}_n \cong \langle a: a^n \rangle$ with kernels $G_n(v)$  and  $G_n(v \circ u)$, respectively. Then $\nu_{v \circ u}\circ \epsilon_{u}=\nu_{v}$, from which it follows that the diagram commutes.
$$
\xymatrix{G_{n}(v)\ar@{^{(}->}[r]^{i_{a,U}}\ar[d]^{\epsilon_u} & E_{n}(v)\ar[d]^{\epsilon_u}\ar[r]^{\nu_{v}} & \mathbb{Z}_{n}\ar@{=}[d]\\
G_{n}(v \circ u)\ar@{^{(}->}[r]^{i_{a,x}} & E_{n}(v \circ u)\ar[r]^{\nu_{v \circ u}} & \mathbb{Z}_{n}}
$$
\end{Lemma}

\begin{proof}  Recall the canonical retractions are defined by $\nu_{v \circ u}(x)=\nu_v(U)=1$ and $\nu_{v \circ u}(a)=\nu_v(a)=a$, where $u \in G_{n}(v \circ u) \cong \mathrm{Ker}(\nu_{v \circ u})$ corresponds to $i_{a,x}(u)\in \mathrm{Ker}(\nu_{v \circ u})$. Then $\nu_{v \circ u}\circ \epsilon_{u}=\nu_{v}$ since $\nu_{v}(U)=1=\nu_{v \circ u}(i_{a,x}(u))=\nu_{v \circ u}(\epsilon_u(U))$ and $\nu_{v}(a)=a=\nu_{v \circ u}(a)=\nu_{v \circ u}(\epsilon_{u}(a))$. \qed
\end{proof}

A simple diagram chase using Lemma \ref{Lemma: epsilon extension commutes} proves the following.

\begin{Theorem} \label{Theorem: injective} Let $n>0$, $u\in F(\mathbf{x})$, and $v \in F(\mathbf{u})$. The substitution homomorphism $\epsilon_u:G_{n}(v) \rightarrow G_n(v \circ u)$ is injective if and only if the induced homomorphism $\epsilon_u: E_n(v) \rightarrow E_n(v \circ u)$ is injective.
\end{Theorem}


\section{The normal closure of $\epsilon(G)$ in $\hat{G}$} \label{Section: The normal closure G in hat G}

We now return to the general situation where we are given a group presentation $\mathcal{G}=\langle \mathbf{u}: \mathbf{v} \rangle$ and substitution homomorphism $\epsilon: F(\mathbf{u}) \rightarrow F(\mathbf{x})$ and investigate the group $\hat{G}$ presented by $\mathcal{P}= \langle \mathbf{x}: \epsilon(\mathbf{v}) \rangle$. Our starting point is Theorem \ref{Theorem: intro short exact} which states that
$$1 \rightarrow \langle \langle \epsilon(G) \rangle \rangle_{\hat{G}} \xrightarrow{i} \hat{G} \xrightarrow{\pi} H \rightarrow 1$$
is an exact sequence of homomorphisms where $H \cong \langle \mathbf{x}: \epsilon(\mathbf{u}) \rangle$.

\subsection{Extensions and subgroups of $F(\textbf{x})$}

Consider now the image of $G$ by the substitution homomorphism $\epsilon$. We are interested in whether $\epsilon(G)$ is a normal subgroup of $\hat{G}$, and, if not, what can be said about its normal closure $N= \langle \langle \epsilon(G) \rangle \rangle_{\hat{G}}$ in $\hat{G}$. To begin, observe that if $\epsilon(G)$ is normal in $\hat{G}$ and additionally if $\epsilon$ is injective, then Theorem \ref{Theorem: intro short exact} implies the group $\hat{G}$ solves the extension problem for $H$ by $G$. We introduce two working examples of finite groups as follows.

\begin{Example} \label{Example: image normal} Let $G= G_2(u_0^2u_1^{-1}) \cong \mathbb{Z}_3$ and consider the substitution $u=x_0^2x_1^{-1}$ where also $H=G_2(u) \cong \mathbb{Z}_3$. We show $\hat{G}=G_2(u^2\theta(u)^{-1})= G_2((x_0^2x_1^{-1})^2x_0x_1^{-2})$ is an extension of $H$ by $G$. The element $u^3=\epsilon_u(u_0^3)=1\in \hat{G}$ so $u^2=u^{-1}$ and we may rewrite the relation $1 = u^2(\theta(u))^{-1}$ as $1=(x_0^2x_1^{-1})^{-1}x_0x_1^{-2}$. This implies $x_0=x_1^{-1} \in \hat{G}$. The subgroup $\epsilon_u(G)\cong \langle u \rangle \cong \langle x_0^3 \rangle$ is nontrivial in $\hat{G}$ since it has nontrivial image under abelianization $\hat{G} \overset{ab}{\rightarrow} \langle x_0 \rangle \cong \mathbb{Z}_9$. Therefore $\epsilon(G) \cong G \cong \mathbb{Z}_3$ and it is normal in $\hat{G}$ since $x_iux_i^{-1}= x_0^{3}=u$ for $i=0,1$. Thus $\hat{G}$ is an extension of $H$ by $G$. In fact, the identity $x_0=x_1^{-1}$ implies $\hat{G}$ is abelian, proving $\hat{G} \cong \mathbb{Z}_9$.
\end{Example}

If $\epsilon(G)$ is not normal in $\hat{G}$, then Theorem \ref{Theorem: intro short exact} provides that $\hat{G}$ modulo the normal closure of $\epsilon(G)$ is isomorphic to $H \cong \langle \mathbf{x}: \epsilon(\mathbf{u}) \rangle$. Consequently, $\epsilon(\mathbf{u})$ normally generates $\hat{G}$ if and only if the substitution homomorphism determines the trivial group $H=1$. Let us consider a finite case when $\hat{G}$ is normally generated by one element. A family of groups $Q(k,l)$ which are normally generated by one element and contains infinite groups will be explored in Section \ref{Section: QKL}.

\begin{Example} \label{Example: Dihedral} Let $n>0$ be odd, $v=u_0u_1$, $u=x_0x_1x_0^{-1}$, and consider the group given by
$$\hat{G}=G_n(v \circ u)=G_n(x_0x_1x_0^{-1}x_1x_2x_1^{-1})$$
where $G= G_n(v) \cong \mathbb{Z}_2$ is generated by $u_0$. The substitution group $H=G_n(u)=1$ so Theorem \ref{Theorem: intro short exact} implies $\hat{G}$ is normally generated by $\epsilon(G)$. Therefore the involution $u = \epsilon(u_0)$ normally generates $\hat{G}$. We have $\hat{G}^{ab} \cong G_n(u_1u_2) \cong \mathbb{Z}_2$ so $\hat{G}$ is nontrivial with $u \neq 1$.

The powers of the shift automorphism of $\hat{G}$ have fixed point $\theta^i(u)=u$ since $u_i=u_0 \in G$, from which it follows that $|u_i|=|x_i|=2$ for $0 \leq i <n$. Therefore the relation $x_ix_{i+1}x_i^{-1}x_{i+1}x_{i+2}x_{i+1}^{-1}=1$ implies $(x_ix_{i+1})^2=x_{i+1}x_{i+2}$, and in turn
$$x_0x_1=(x_{n-1}x_{0})^2=\dots=(x_0x_1)^{2^n}$$
so that $(x_0x_1)^{m}=1$ for $m=2^n-1$. Now the relation $(x_ix_{i+1})^2=x_{i+1}x_{i+2}$ also provides that $x_{i+2}=x_{i+1}^{-1}(x_ix_{i+1})^2$, and solving the recurrence relation we have $x_i= x_0(x_0x_1)^{2^i-1}$ for $0<i<n$. Hence $\hat{G}$ is generated by $x_0$ and $x_0x_1$. These generators satisfy $x_0(x_0x_1)x_0(x_0x_1)= 1$, and so it follows that $\hat{G}$ is isomorphic to the dihedral group $D_m$ of order $2m$. Indeed, for $r,s$ the generators of $D_m$ such that $s^2=r^m=srsr=1$, mappings are given by $s \leftrightarrow x_0$, $r \leftrightarrow x_0x_1$, and $x_i=x_0(x_0x_1)^{2^i-1} \leftrightarrow sr^{2^i-1}$ for $0<i<n$.
\end{Example}

The difference between $\epsilon(G)$ and its normal closure in $\hat{G}$ can be stated in terms of subgroups of $F(\mathbf{x})$. We may then leverage the fact that the subgroup of $F(\mathbf{x})$ generated by $\epsilon(\mathbf{u})$ contains $\epsilon(\mathbf{v})$. This gives access to index arguments that may be useful in determining $\hat{G}$, and which may also be implemented computationally using GAP. For brevity, we will write $\epsilon(F(\mathbf{u}))$ for the subgroup generated by $\epsilon(\mathbf{u})$ in $F(\mathbf{x})$.

\begin{Proposition} \label{Prop: free subgroups}  Let $\epsilon: F(\mathbf{u}) \rightarrow F(\mathbf{x})$ be a homomorphism and $G \cong \langle \mathbf{u}: \mathbf{v} \rangle$ with $\mathbf{v} \subseteq F(\mathbf{u})$. Then for $R_{\mathbf{v}}= \langle \langle \mathbf{v} \rangle \rangle_{F(\mathbf{u})}$, there are isomorphisms $\epsilon(G) \cong \epsilon(F(\mathbf{u})) \big/ \epsilon(R_{\mathbf{v}})$ and $\langle \langle \epsilon(G) \rangle \rangle_{\hat{G}} \cong \langle \langle \epsilon(\mathbf{u}) \rangle \rangle_{F(\mathbf{x})} \big/ \langle \langle \epsilon(\mathbf{v}) \rangle \rangle_{F(\mathbf{x})}$, and the index of $\epsilon(G)$ in its normal closure $\langle \langle \epsilon(G) \rangle \rangle_{\hat{G}}$ equals the index of $\epsilon(F(\mathbf{u}))\langle \langle \epsilon(\mathbf{v}) \rangle \rangle_F$ in $\langle \langle \epsilon(\mathbf{u}) \rangle \rangle_{F(\mathbf{x})}$.
\end{Proposition}

\begin{proof} The subgroup generated by $\epsilon(\mathbf{v})$ is contained in the subgroup generated by $\epsilon(\mathbf{u})$ in $F(\mathbf{x})$, where for $R_\mathbf{v} = \langle \langle \mathbf{v} \rangle \rangle_{F(\mathbf{u})}$ there is the subgroup lattice
\begin{align*}
\begin{xymatrix}{
&& \langle \langle \epsilon(\mathbf{u}) \rangle \rangle_{F(\mathbf{x})} \ar@{-}[d]\\
&& \epsilon(F(\mathbf{u}))\langle \langle \epsilon(\mathbf{v}) \rangle \rangle_{F(\mathbf{x})} \ar@{-}[dl]\ar@{-}[dr]\\
& \epsilon(F(\mathbf{u})) \ar@{-}[dr] & & \langle \langle \epsilon(\mathbf{v}) \rangle \rangle_{F(\mathbf{x})} \ar@{-}[dl] \\
&& \epsilon(R_\mathbf{v})}
\end{xymatrix}
\end{align*}
with $\epsilon(F(\mathbf{x})) \cap \langle \langle \epsilon(\mathbf{v}) \rangle \rangle_{F(\mathbf{x})} \cong \langle \langle \epsilon(\mathbf{v}) \rangle \rangle_{\epsilon(F(\mathbf{u}))} \cong \epsilon (R_\mathbf{v})$. Let $R_{\epsilon(\mathbf{v})}$ denote the normal closure of $\epsilon(\mathbf{v})$ in $F(\mathbf{x})$, where the image of $G$ is isomorphic to $\epsilon(F(\mathbf{u})) R_{\epsilon(\mathbf{v})}  \big/ R_{\epsilon(\mathbf{v})}$. Thus by the second isomorphism theorem, $\epsilon(G) \cong \epsilon(F(\mathbf{u})) \big/ \epsilon(R_\mathbf{v})$. Finally, the normal closure of $\epsilon(F(\mathbf{u})) R_{\epsilon(\mathbf{v})}$ in $F(\mathbf{x})$ is $\langle \langle \epsilon(\mathbf{u}) \rangle \rangle_{F(\mathbf{x})}$ since $R_{\epsilon(\mathbf{v})} \leq \langle \langle \epsilon(\mathbf{u}) \rangle \rangle_{F(\mathbf{x})}$, so $\langle \langle \epsilon(G) \rangle \rangle_{\hat{G}} \cong \langle \langle \epsilon(\mathbf{u}) \rangle \rangle_{F(\mathbf{x})} / R_{\epsilon(\mathbf{v})}$.
\qed
\end{proof}

Let us illustrate the situation using the preceding examples. Example \ref{Example: image normal} has $\epsilon(G) \cong \mathbb{Z}_3$ a normal subgroup of $\hat{G} \cong \mathbb{Z}_9$. Working in $F(\mathbf{x})$, with $\bar{N}= \langle \epsilon(\mathbf{u}) \rangle$ and $R=\langle \langle \epsilon(\mathbf{v}) \rangle \rangle_{F(\mathbf{x})}$,
\begin{align*}
x_0 u x_0^{-1} u^{-1} ((\theta(u))^2 u^{-1})^{-1}&= x_0^3x_1^{-1}x_0^{-1}(\theta(u))^{-2} \\
& = x_0^3x_1^{-3}(\theta(u))^{-1}
\end{align*}
is in $\bar{N}R$, so $x_0 u x_0^{-1}u^{-1} \in \bar{N}R$, too. Therefore $x_0$ commutes with $u$ modulo $\bar{N}R$. After similarly accounting for commutivity among $x_1$ and $\theta(u)$, it follows that $N= \langle \langle \epsilon(\mathbf{u}) \rangle \rangle_{F(\mathbf{x})}$ has only one coset modulo $\bar{N}R$. In Example \ref{Example: Dihedral}, where $\epsilon(G) \cong \mathbb{Z}_2$ and $\hat{G}$ is the dihedral group of order $2m$ for $m=2^n-1$, the left cosets of $R$ in $N= F(\mathbf{x})$ (since $H=1$) are represented by $x_0^{\delta}(x_0x_1)^j$, for $0 \leq j< 2^n-1$ and $\delta=0,1$. Passing to $\bar{N}R$, the left cosets in $F(\mathbf{x})$ are of the form $(x_0x_1)^{j} \bar{N}R$ for $0 \leq j< 2^n-1$. This is because for $\delta=0$ the coset 
$$(x_0x_1)^j \bar{N}R= (x_0x_1)^j (x_0x_1x_0^{-1})(x_0^2)\bar{N}R=(x_0x_1)^{j+1}x_0 \bar{N}R= x_0(x_0x_1)^{m-(j+1)}\bar{N}R$$
for $0<j<2^n-1$. 

\subsection{Finiteness problem for $\hat{G}$} \label{Subsection: finiteness}

Consider now the finiteness problem for the group $\hat{G}$. Our examples will show it is possible for $\hat{G}$ to be finite when both $G$ and $H$ are nontrivial and finite. However, this is rarely the case, for $\hat{G}$ is infinite if additionally $\hat{\mathcal{G}}$ is aspherical by Theorem \ref{Theorem: relative aspherical}. Experiments in GAP show that even if $G$ is finite cyclic and $H=1$, then typically $\hat{G}$ is infinite. See Example \ref{Example: normally generated by involutions} for a specific instance. 

We begin this section with some simple consequences of Theorem \ref{Theorem: intro short exact} in terms of the finiteness problem and a couple of remarks. After that, we will discuss those finite cyclically presented groups found from a computational search. The list of finite groups recorded in Figure \ref{Table: 1} are the results of the search. Additional finite groups $\hat{G}$ appear in Subsection \ref{Subsection: iterates} and Lemma \ref{Lemma: GL23}.

\begin{Corollary} \label{Corollary: quotient} Let $\hat{G}$ be a group given by a composite presentation $\mathcal{P}=\langle \mathbf{x}: \epsilon(\mathbf{v}) \rangle$, with $G \cong \langle \mathbf{u}: \mathbf{v} \rangle$, $H \cong \langle \mathbf{x}: \epsilon(\mathbf{u}) \rangle$ and substitution homomorphism $\epsilon: F(\textbf{u})\rightarrow F(\textbf{x})$.
\begin{itemize} 
\item[(i)]  If $G=H=1$, then $\hat{G}=1$.
\item[(ii)]  If $H$ is infinite (resp. SQ-universal), then $\hat{G}$ is infinite (resp. SQ-universal).
\item[(iii)]  If $H=1$, then $\hat{G}$ is normally generated by $\mathrm{Im}(\epsilon)$. In particular, $\hat{G}$ is finitely generated by conjugates of the substitutions $\{ \epsilon(u_0), \epsilon(u_1), \dots, \epsilon(u_{k-1}) \} \subseteq \hat{G}$.
\item[(iv)] If $G$ is infinite and $\epsilon$ is injective, then $\hat{G}$ is infinite.
\end{itemize}
\end{Corollary}

With regards to Corollary \ref{Corollary: quotient} [i], it is well known that $\hat{G}=1$ if $G=H=1$ are trivial (see Section \ref{Section: background} for references). More generally, it may be possible for $\hat{G}$ to be trivial with nontrivial $G$, in which case $\epsilon$ fails to be injective (and is the trivial map) and $H=1$. Recall that $\hat{G}$ is obtained from $G$ by adjoining the generators $\textbf{x}$ and relations $\textbf{u}=\epsilon(\textbf{u})$ to $\mathcal{G}=\langle \textbf{u}: \textbf{v} \rangle$ by Lemma \ref{Lemma: relative hat G}. From this perspective, we are considering the problem of whether an adjunction of $G$ may result in the trivial group. We recall the Kervaire conjecture which states that if $G$ is any nontrivial group, then $G \ast F(x) / \langle \langle w \rangle \rangle $ is nontrivial for any $w \in G \ast F(x)$ (where $x$ is not a symbol in $G$). Thus if a group presentation is partially composite with one adjunction (e.g. Example \ref{Example: partial comp}), with the adjunction such that $H=1$, then such an example may be used to test the Kervaire conjecture. The shift extension of a cyclically presented group having composite relations provides a source for such presentations (see Lemma \ref{Lemma: E relative}). However, shift extensions always retract onto $\mathbb{Z}_n$, so a counterexample to the Kervaire conjecture cannot be obtained in this way. On the other hand, this does not rule out the possibility that the cyclically presented group $G_n(v \circ u)$ itself is trivial; here $G_n(v \circ u)$ is obtained from $G_n(u)$ by $n$ adjunctions.

\begin{Example} It is possible for $\epsilon(G)$ to be trivial even when $G$ and $\hat{G}$ are not. Consider the presentation $\mathcal{G}=\mathcal{G}_3(u_0u_1)$ for $G \cong \mathbb{Z}_2$. Define $\epsilon$ by $u_i \mapsto x_ix_{i+1}^{-1}$ for $0\leq i<3$. Then $\hat{G}$ has relators $\epsilon(u_iu_{i+1})= x_ix_{i+1}^{-1}x_{i+1}x_{i+2}^{-1}=x_ix_{i+2}^{-1}$ for $0\leq i <2$ so that $x_0=x_1=x_2$ in $\hat{G} \cong \mathbb{Z}$. The generator of $G$ maps to $\epsilon(u_0)= x_0x_1^{-1}=1$ and so $\epsilon(G)=1$.
\end{Example}

Finiteness of a group $\hat{G}$ has the following interpretation in terms of homomorphisms $\epsilon: F(\textbf{u})\rightarrow F(\textbf{x})$ and subgroups of $F(\textbf{x})$. Let $\epsilon: F(\textbf{u}) \rightarrow F(\textbf{x})$. If $H \cong \langle \textbf{x}: \epsilon(\textbf{u}) \rangle$ is finite, then it is possible for there to exist $\textbf{v} \subseteq F(\textbf{u})$ with finite $\hat{G}$, in which case $F(\textbf{x}) / \langle \langle \epsilon(\textbf{v}) \rangle \rangle$ is finite. That is, there may exist a subgroup $\langle \textbf{v} \rangle \leq F(\textbf{u})$ whose image normally generates a subgroup of finite index in $F(\textbf{x})$. On the other hand, if $H$ is infinite, then the image of every subgroup of $F(\textbf{u})$ normally generates a subgroup of infinite index in $F(\textbf{x})$ by Corollary \ref{Corollary: quotient} [ii]. Thus when $H$ is infinite, $\epsilon$ fails to be surjective in a strong sense. Compare with Lemma \ref{Lemma: finite iterate} for a substitution which fails to be surjective in a weak sense.

\textbf{Computational Search} The finite groups obtained from a selective, non- exhaustive search of composite presentations of the form $G_n(v \circ u)$ are recorded in Figure \ref{Table: 1}. Such a search has considerable computational limitations for the word length of the relation $v \circ u$ is $|v||u|$ in most cases. Search candidates included certain small finite groups $G_n(v)$ with word length $|v| \leq 3$ and $n\leq 5$ (e.g. cyclic, $Q_8$, $SL(2,3)$, $SL(2,5)$), and for $G_n(u)$ the same selection of groups as well as well-known cyclic presentations of the trivial group. This search therefore depended on the solution to the finiteness problem for cyclically presented groups $G_n(w)$ with word length less than or equal to 3; for the purpose of our informal census we cite \cite{EdjWillThree} and \cite{WillFibRev} and refer the interested reader to the references therein. Also, note that the case of odd $n$ with $v=u_0u_1$, $u=x_0x_1$, and $G_n(v \circ u) \cong \mathbb{Z}_4$ was already identified in \cite{BogParFour} in solving the finiteness problem for cyclically presented groups $G_n(w)$ where $w$ has positive exponents and length $4$. To the best of the author's knowledge, all other presentations for finite $\hat{G}$ are new.

The search itself was performed using GAP \cite{GAP} and standard tests for finiteness. Here one may use the Size command to identify finite groups, and AbelianInvariants to identify those groups having infinite abelianizations. In case the abelianization is finite, it is possible to perform Newman's infinity criterion \cite{NewInf} implemented via NewmanInfinityCriterion in order to prove the group $\hat{G}$ is infinite. Additionally, it may be possible to work further down the derived series for $\hat{G}$ and repeat these tests. Newman's criterion applied to $\hat{G}'$ and prime divisors of $|\hat{G}' \big/ \hat{G}''|$ frequently succeeded in proving $\hat{G}$ is infinite, so was regularly implemented within a larger ad hoc scheme to solve the finiteness problem for $\hat{G}$. See Example \ref{Example: normally generated by involutions}. The success in working with $\hat{G}'$ is likely traced to Proposition \ref{Prop: free subgroups}, where the subgroup $\epsilon(\mathbf{u}) \langle \langle \epsilon(\mathbf{v}) \rangle \rangle_{F(\mathbf{x})}$ of $F(\mathbf{x})$ contains the subgroup $[\bar{N}, F(\mathbf{x})]$ of the commutator subgroup of $F(\mathbf{x})$, for $\bar{N}$ the subgroup of $F(\mathbf{x})$ generated by $\epsilon(\mathbf{v})$. Finally, there is the StructureDescription command which can be used to identify small finite groups; this was successful for all finite groups in this paper.

\begin{Example} \label{Example: normally generated by involutions} Let $v=u_0 u_1$, $u=x_0 x_3 x_{0}^{-1} x_3^{-2}$, and consider $Q=G_9(v \circ u)$ where $v \circ u =x_0x_3x_0^{-1}x_3^{-2}x_1x_4x_1^{-1}x_4^{-2}$. Then $G_9(v) \cong \langle u_0 \rangle \cong \mathbb{Z}_2$, $G_9(u)=1$, and $Q$ is normally generated by $\epsilon_u(u_0)=u$. The group $Q$ is infinite since its commutator subgroup $Q'$ has the quotient $Q'/Q'' \cong (\mathbb{Z}_7)^3$, and $Q'$ satisfies Newman's infinity criterion \cite{NewInf} with $p=7$. Thus $Q$ is an infinite group that is generated by conjugates of $u$, all involutions. 
\end{Example}

\begin{figure}
\centering
\begin{tabular}{ |c|c|c|c|c|c| } 
\hline
 $n>0$ & $u \in F(\mathbf{x})$ & $H=G_n(u)$ & $v \in F(\mathbf{u})$ & $G=G_n(v)$ & $\hat{G}=G_n(v \circ u)$  \\ 
\hline
 odd $n$ & $x_0x_1x_0^{-1}$ & $1$ & $u_0u_1$ & $\mathbb{Z}_2$ & $D_{m}, m=2^n-1$  \\  
\hline
 $2$ & $x_0x_1x_0^{-1}$ & $1$ & $u_0^2u_1$ & $\mathbb{Z}_3$ & $\mathbb{Z}_7 \rtimes \mathbb{Z}_3$  \\  
 \hline
 $2$ & $x_0x_1x_0^{-1}$ & $1$ & $u_0^2u_1^{-1}$ & $\mathbb{Z}_3$ & $SL(2,3)$  \\ 
 \hline
 odd $n$ & $x_0x_1$ & $\mathbb{Z}_2$ &  $u_0u_1$ & $\mathbb{Z}_2$ & $\mathbb{Z}_4$\\ 
 \hline
 $2$ & $x_0^2x_1$ & $\mathbb{Z}_3$ & $u_0^2u_1$ & $\mathbb{Z}_3$ & $\mathbb{Z}_{19} \rtimes \mathbb{Z}_9 $ \\ 
 \hline
 $2$ &  $x_0^2x_1^{-1}$ & $\mathbb{Z}_3$ & $u_0^2u_1$ & $\mathbb{Z}_3$ & $\mathbb{Z}_9 \rtimes \mathbb{Z}_3$ \\ 
 \hline
 $2$ &  $x_0^2x_1^{-1}$ & $\mathbb{Z}_3$ & $u_0^2u_1^{-1}$ & $\mathbb{Z}_3$ & $\mathbb{Z}_9$ \\ 
 \hline
\end{tabular} \\
\caption{Finite groups having composite presentations of the form $\mathcal{G}_n(v \circ u)$.}
\label{Table: 1}
\end{figure}

\textbf{Results} We summarize and interpret the finite groups identified from our search as follows. Concerning normality, two finite cases with $\hat{G} \cong \mathbb{Z}_4$ and $\mathbb{Z}_9$ have the property that $\epsilon(G)$ is normal in $\hat{G}$; except for certain trivial cases with $\epsilon$ an isomorphism of free groups, no other examples (including infinite ones) were identified where $\epsilon(G)$ is normal in $\hat{G}$. 

There are two infinite families of composite presentations which are finite for odd $n$; that they are in fact finite for all odd $n>0$ follows from Example \ref{Example: Dihedral} and \cite[Theorem 7.2]{BogParFour} (see also Lemma \ref{Lemma: another finite iterate}). In all other cases, we were unable to find a finite group with $n>2$ either because the group under consideration was infinite or the GAP calculations were inconclusive. 

The two cases with $n=2$, $u=x_0x_1x_0^{-1}$, and $G_2(v) \cong \mathbb{Z}_3$ produce different groups $\hat{G}$; thus $\hat{G}$ is not an invariant of the pair $(G, \epsilon )$. Similarly, $\hat{G}$ is not an invariant of $\mathcal{G}$ and $H$, as can be seen by the examples with $\mathcal{G}=\mathcal{G}_n(u_0^2u_1)$ and $H \cong \mathbb{Z}_3$.

\subsection{Groups of weight one} \label{Subsection: normal closure}

Recall that the weight of a group is the least number of elements required to normally generate the group, where the trivial group is said to have weight zero. Thus if $G$ is a group of weight one and the substitution $\epsilon$ is such that $H=1$, then $\hat{G}$ has weight one or is trivial. The latter scenario may be avoided if we restrict ourselves to those groups $G$ given by a presentation $\mathcal{G}$ of deficiency one, where the deficiency of a group presentation is the number of generators minus the number of relators, and the deficiency of a group is the maximum over all presentations. The class of groups obtained by substitutions of the trivial group into all presentations of deficiency one is still very general. Thus we will further simplify the situation by considering only those presentations $\mathcal{G}$ in the family of presentations $\mathcal{D}=\mathcal{D}(\mathbb{Z}, 1)$ for $\mathbb{Z}$ having deficiency one. Define $\mathcal{B}=\mathcal{B}(\mathcal{D}, 1)$ to be the set of all composite group presentations obtained by composing a presentation $\mathcal{G} \in \mathcal{D}$ with a balanced presentation $\mathcal{H}$ of the trivial group whose number of generators equals that of $\mathcal{G}$.

\begin{Lemma} \label{Lemma: Z by 1} Every presentation $\mathcal{P} \in \mathcal{B}$ defines a group $\hat{G}$ with weight one, infinite cyclic abelianization, and deficiency one. Furthermore, the natural map $\epsilon: \mathbb{Z} \rightarrow \hat{G}$ is injective.
\end{Lemma}

\begin{proof} Let $\mathcal{P} \in \mathcal{B}$ and $\mathcal{G}=\langle \textbf{u}: \textbf{v} \rangle$ the corresponding presentation for $G \cong \mathbb{Z}$ of deficiency one. Let $\mu \in F(\textbf{u})$ be a generator of $G \cong \mathbb{Z}$. Then by Theorem \ref{Theorem: intro short exact}, $\hat{G}$ is normally generated by $\epsilon(\mu)$. Now $\mathcal{P}$ has deficiency one and $\hat{G}$ has weight one, so $\hat{G}$ has deficiency one. Then $\hat{G}$ has weight one and deficiency one which implies $\hat{G}^{ab} \cong \mathbb{Z}$ (e.g. \cite[p. 17]{Hill2andgroup}). Consequently, $\epsilon(\mu)$ has infinite order in $\hat{G}$ from which it follows that $\epsilon: G \rightarrow \hat{G}$ is injective. 
\qed
\end{proof}

We identify the class of groups $\mathcal{G}(\mathcal{B})$ defined by presentations in $\mathcal{B}$ as follows. For $n>0$, let $K_n$ denote the class of all $n$-knot groups, where a group is an $n$-knot group if it is the fundamental group of the complement of an embedding of $S^n$ in $S^{n+2}$. It is well known that $K_1$ is a proper subset of $K_2$ and we recall two additional facts as follows. First, every classical knot group $G \in K_1$ admits a deficiency one Wirtinger presentation, i.e. each relation is of the form $x_j=w_jx_0w_j^{-1}$ where $w_j \in F(\textbf{x})$ and $0<j<|\textbf{x}|$ (e.g. \cite[Chapter 3]{RolfKnot}). Second, that groups satisfying the conclusion of Lemma \ref{Lemma: Z by 1} are 2-knot groups. This is a consequence of the Kervaire conditions \cite{KerCond}, which state that the class of $n$-knot groups for $n\geq 3$ are precisely those finitely presented groups with weight one, infinite cyclic abelianization, and trivial second homology. The requirement for a deficiency one presentation in defining $\mathcal{B}$ is stronger than the homology condition and sufficient for the groups to be 2-knot groups. For details and additional references see \cite[Section 14.6]{Hill2Survey}. 

\begin{Theorem} \label{Theorem: Z by 1} The containments $K_1 \subset \mathcal{G}(\mathcal{B}) \subset K_2$ are all proper.
\end{Theorem}

The following example is of a group in $\mathcal{G}(\mathcal{B})$ which is not a 1-knot group, and we prove the remainder of the theorem thereafter.

\begin{Example} \label{Example: T bar not K1} The group $\bar{T}$ given by 
$$\bar{\mathcal{T}} = \langle x_0, x_1: x_0x_1x_0^{-1}x_1^{-2}=x_1x_0x_1^{-1}x_0^{-2} \rangle $$
has $\bar{\mathcal{T}} \in \mathcal{B}$ since it equals the substitution on $\langle u_0, u_1: u_0=u_1 \rangle$ for $\mathbb{Z}$ given by $u_i \mapsto x_ix_{i+1}x_i^{-1}x_{i+1}^{-2}$ for $i=0,1$, where the substitution group $H=1$ is given by Higman's presentation of the trivial group on $n=2$ generators. We show the Alexander polynomial associated to $\bar{T}$ is not symmetric, from which it follows that $\bar{T}$ is not the group of a 1-knot (e.g. \cite[7D, 8C]{RolfKnot}).

Writing the relation as a relator and conjugating by $x_0$ gives
$$\bar{\mathcal{T}} \cong \langle x_0, x_1: x_0^2x_1x_0^{-1}x_1^{-2}x_0^2x_1x_0^{-1}x_1^{-1}x_0^{-1} \rangle$$
and adjoining the generator $a$ by the relation $a=x_1x_0^{-1}$ and then eliminating $x_1$ gives
\begin{align*}
\bar{\mathcal{T}} & \cong \langle x_0, a: x_0^2(ax_0)x_0^{-1}(ax_0)^{-2}x_0^2(ax_0)x_0^{-1}(ax_0)^{-1}x_0^{-1} \rangle \\
& = \langle x_0, a: (x_0^2ax_0^{-2})(x_0a^{-1}x_0^{-1})a^{-1}(x_0^2ax_0^{-2})(x_0a^{-1}x_0^{-1}) \rangle
\end{align*}
so that the polynomial is $p(t)=2t^2-2t-1$ where $p(t) \neq p(1/t)$.
\end{Example}

\begin{proof} To see $K_1 \subseteq \mathcal{G}(\mathcal{B})$, let $\mathcal{P}$ be a Wirtinger presentation of deficiency 1, i.e. $\mathcal{P}=\langle x_0, x_1, \dots, x_{n-1}: w_j^{-1}x_jw_j=x_0, \ 0<j<n \rangle$ with each $w_j \in F(\textbf{x})$. Define $\mathcal{G}= \langle u_0, u_1, \dots, u_{n-1}: u_j=u_0, \ 0<j<n \rangle$ and $\epsilon:F(\textbf{u})\rightarrow F(\textbf{x})$ by $u_0 \mapsto x_0$ and $u_j \mapsto  w_j^{-1}x_jw_j$ for $0<j<n$. Then $\mathcal{G}$ is a presentation for $\mathbb{Z}$ and $\mathcal{H}=\langle x_0, x_1, \dots, x_{n-1}: x_0, w_j^{-1}x_jw_j, \ 0<j<n \rangle$ is a presentation for the trivial group. Thus $\mathcal{P}  \in \mathcal{B}$. The fact that $K_1$ is a proper subset of $\mathcal{G}(\mathcal{B})$ follows from Example \ref{Example: T bar not K1}. For the other containment, Lemma \ref{Lemma: Z by 1} provides that $ \mathcal{G}(\mathcal{B}) \subseteq K_2$. The subset is proper since there exist 2-knot groups having deficiency not equal to one \cite[I.2]{KerCond}.
\qed
\end{proof}

While the presentations $\mathcal{B}$ use substitutions corresponding to presentations of the trivial group, compare with \cite{RatRoot} and \cite{LieRoot} who adjoin roots to a meridian of a 1-knot group - and so the substitution group is finite cyclic- and also obtain 2-knot groups which are not 1-knot groups. Notice too that if we vary $\mathcal{D}=\mathcal{D}(\mathbb{Z},1)$ and maintain $H=1$, then we arrive at classes of groups related to the fundamental group of other codimension 2 smooth embeddings (see \cite[Section 2]{GonGorSimUnsol}).

Recall that the class of groups having infinite cyclic abelianization and which are given by a deficiency one Wirtinger presentation are precisely the class of ribbon 2-knot groups \cite{YajWirtKnot}. A group is a ribbon 2-knot group if it is the fundamental group of the complement of a smooth embedding of a so-called ribbon disk $D\rightarrow D^4$. Thus in proving $K_1 \subseteq \mathcal{G}(\mathcal{B})$ we have actually proved the following.

\begin{Corollary} $\mathcal{G}(\mathcal{B})$ contains the class of ribbon 2-knot groups.
\end{Corollary}

It could not be determined if this containment is proper. One method for doing so is to take a deficiency one presentation $\mathcal{P} \in \mathcal{B}$ and adjoin the relator $\epsilon(\mu)$, for $\mu \in F(\textbf{u})$ a generator of $G\cong \mathbb{Z}$, to obtain a balanced presentation $\mathcal{Q}$ for the trivial group. Then $\hat{G}$ presented by $\mathcal{P}$ is a ribbon 2-knot group if and only if $\mathcal{Q}$ is Andrews- Curtis equivalent to the presentation $\langle \textbf{x}: \textbf{x} \rangle$ \cite{YoshNote}. In this way, we can show the group $\bar{T}$ from Example \ref{Example: T bar not K1} is a ribbon 2-knot group. For if we adjoin the relator $u$ to its presentation, then the resulting presentation is AC-equivalent to $\langle x_0, x_1: u, \theta(u) \rangle$ which in turn is known to be AC-equivalent to $\langle x_0, x_1: x_0, x_1 \rangle$. 

On the other hand, consider the group $\bar{T}_3$ presented by $\langle x_0, x_1, x_2: u=\theta(u), \theta(u)=\theta^2(u) \rangle$ in $\mathcal{B}$ still with $u=x_0x_1x_0^{-1}x_1^{-2}$. Adjoining the relation $u=1$ to this presentation is AC-equivalent to $\langle x_0, x_1, x_2: u, \theta(u), \theta^2(u) \rangle$, where it is an open problem to determine if this presentation of the trivial group is AC-equivalent to $\langle x_0, x_1, x_2: x_0, x_1, x_2  \rangle$ (this is referred to as Neumann's potential counterexample in the literature  \cite{BurBal}). Thus if Neumann's example is a counterexample to the AC-conjecture, then it would also provide that the class of groups $\mathcal{G}(\mathcal{B})$ properly contains the ribbon 2-knot groups. For more on the groups $\bar{T}$ and $\bar{T}_3$, see Section \ref{Section: QKL} where we introduce and study the groups $Q(k,l)$. We show $Q(2,1)\cong T$ is the group of the trefoil knot and $Q(3,1) \cong \bar{T}_3 / \langle \langle u^2 \rangle \rangle$, i.e. $Q(3,1)$ is obtained from $\bar{T}_3$ by adjoining $u^2=1$ rather than $u=1$.

\subsection{Iterated substitutions $\Phi^n$} \label{Subsection: iterates}

Substitutions have the following semigroup structure. Given $\epsilon: F(\mathbf{u}) \rightarrow F(\mathbf{x})$ and $\delta: F(\mathbf{x}) \rightarrow F(\mathbf{t})$ for $\mathbf{t}=\{ t_0, t_1, \dots, k_{l-1} \}$ and $l>0$, the composition $\psi= \delta \circ \epsilon: F(\mathbf{u}) \rightarrow F(\mathbf{t})$ is a substitution homomorphism and the induced maps
$$G \overset{\epsilon}{\rightarrow} \hat{G} \overset{\delta}{\rightarrow} \hat{\hat{G}}$$
are such that $\psi = \delta \circ \epsilon$, too. When there are $k=n$ substitutions, there is the identity substitution $\iota: F(\mathbf{u}) \rightarrow F(\mathbf{x})$ given by $u_i \mapsto x_i$ for $0 \leq i < n$, and the substitution $\epsilon$ determines the endomorphism $\Phi_{\epsilon}= \epsilon \circ Id$ of $F(\mathbf{x})$, for $Id: F(\mathbf{x}) \rightarrow F(\mathbf{u})$ defined by $x_i \mapsto u_i$ for $0 \leq i< n$. Thus the set of substitution homomorphisms having $k=n$ substitutions are identically the set of endomorphisms of $F(\mathbf{x})$, and therefore is a monoid with identity $\Phi_\iota$. 

Let $\Phi_\epsilon$ be an endomorphism  of $F(\mathbf{x})$. In this situation, we will omit reference to $\mathbf{u}$ and write $\mathbf{v}$ in $F(\mathbf{x})$ so that $G \cong \langle \mathbf{x}: \mathbf{v} \rangle$. Here we may compose $\Phi$ with itself and form the sequence of groups
$$G=\hat{G}_0 \rightarrow \hat{G}_1 \rightarrow \hat{G}_2 \rightarrow \dots$$
where $\hat{G}_a \cong \langle \mathbf{x}: \Phi^a(\mathbf{v}) \rangle$ and there is an induced map $\Phi: \hat{G}_a \rightarrow \hat{G}_{a+1}$ for all $a>0$. Repeated applications of Theorem \ref{Theorem: intro short exact} and Proposition \ref{Prop: free subgroups} for $\Phi^a(\mathbf{v})$ relative to the substitution $\Phi^i(\mathbf{x})$ for $0 \leq i<a$ may be used to provide the following structure descriptions of $\hat{G}_a$ and the corresponding subgroup structure in $F(\mathbf{x})$. 

\begin{Theorem} \label{Theorem: endo short exact} For $n>0$, let $\Phi$ be an endomorphism of $F=F(\mathbf{x})$, and $\mathbf{v} \subseteq F$ with $R_i$ the normal closure of $\Phi^i(\mathbf{v})$ in $F$ for $0 \leq i<n$. Then the normal subgroups $N_i$ generated by $\Phi^i(\mathbf{x})$ for $i \geq 0$ form a descending chain of subgroups in $F$, where $R_i \unlhd N_i$ for each $i\geq 0$. Thus for each group $\hat{G}_a \cong \langle \mathbf{x}: \Phi^a(\mathbf{v}) \rangle \cong F \big/ R_a$ with $a > 0$, $\Phi$ induces natural maps $\Phi^{a-i}: \hat{G}_i \rightarrow \hat{G}_{a}$ for $0\leq i < a$, where for $H_i=\langle \mathbf{x}: \Phi^{a-i}(\mathbf{x}) \rangle \cong F \big/ N_{a-i}$, 
$$1 \rightarrow \langle \langle \Phi^{a-i}(\hat{G}_{i}) \rangle \rangle_{\hat{G}_a} \xrightarrow{\iota} \hat{G}_a \xrightarrow{\pi} H_i \rightarrow 1$$
is an exact sequence of homomorphisms. In particular, $\Phi^{a-i}(\hat{G}_i) \cong \Phi^{a-i}(F) R_a \big/ R_a \cong \Phi^{a-i}(F) \big/ \Phi^{a-i}(R_{i}) $ and $\langle \langle \Phi^{a-i}(G_{i}) \rangle \rangle_{\hat{G}_a} \cong N_{a-i} / R_a$ for $0 \leq i< a$.
\end{Theorem}

\begin{proof} Let $\Phi \in \mathrm{End} (F)$. Then any word in $\Phi^{a}(\mathbf{x})=\Phi^{i}(\Phi^{a-i}(\mathbf{x}))$ is composite in $\Phi^{a-i}(\mathbf{x})$ so $\langle \Phi^a(\mathbf{x}) \rangle_F \leq \langle \Phi^{a-i}(\mathbf{x}) \rangle_F$ for $0 \leq i<a$. Therefore $N_a \leq N_{a-i}$, and in particular $N_{i+1} \leq N_{i}$ for $i \geq 0$, so $\Phi$ determines a descending chain of normal subgroups in $F$. Since $\mathbf{v} \subset F(\mathbf{x})$, we also have $\langle \Phi^i(\mathbf{v}) \rangle_F \leq \langle \Phi^i(\mathbf{x}) \rangle_F$ so that $R_i \leq N_i$ for $i \geq 0$. Thus $R_a \leq N_a \leq N_{a-i}$ where $R_a \unlhd N_{a-i}$, so there is the exact sequence 
$$1 \rightarrow R_a \big/ N_{a-i} \rightarrow F/ R_a \rightarrow F \big/ N_{a-i}\rightarrow 1$$
of homomorphisms, where $\hat{G}_a \cong F \big/ R_a$ and $H_i \cong F \big/ N_{a-i}$ by definition.

Let $\bar{N}_i=\langle \Phi^i(\mathbf{x}) \rangle_F$ for $i \geq 0$, i.e. $\bar{N}_i \cong \Phi^i(F)$. There is the subgroup lattice
\begin{align*}
\begin{xymatrix}{
&& \bar{N}_{a-i}R_a \ar@{-}[dl]\ar@{-}[dr]\\
&\bar{N}_{a-i} \ar@{-}[dr] & & R_a \ar@{-}[dl] \\
&&\bar{N}_{a-i} \cap R_a}
\end{xymatrix}
\end{align*}
where $\bar{N}_{a-i} \cap R_a \cong \langle \langle \Phi^a(\mathbf{v}) \rangle \rangle_{\bar{N}_{a-i}} \cong \Phi^{a-i}(R_{i})$ and the remainder of the proof follows as in Proposition \ref{Prop: free subgroups}. 

\qed
\end{proof}

In the introduction, we discussed two references to the presentations obtained by successively composing Higman's presentation of the trivial group with itself. In another direction, we recall a result of Higman stated in terms of substitutions by B. H. Neumann in \cite{NeumEssay}. Let $N$ denote the subgroup of $F(x_0, x_1, x_2)$ normally generated by $r=x_0^{-1}x_1x_0x_1^{-2}$ and $s=x_1^{-1}x_2^{-1}x_1x_2$ (these are the relators for Higman's  first example of a finitely presented non-Hopfian group \cite{HigHop}). Let $\Phi$ denote the free group automorphism of $F(x_0, x_1, x_2)$ which fixes $x_0$ and $x_2$ while mapping $x_1 \mapsto x_0x_1x_0^{-1}$. Finally, denote by $N_a$ the normal closure of $\{ \Phi^a(r), \Phi^a(s) \}$ in $F(\textbf{x})$. Then $N \subset N_1 \subset N_2 \subset \dots$ is a properly ascending infinite chain of normal subgroups in $F(x_0, x_1, x_2)$.

To conclude this section we demonstrate substitutions where $\hat{G}_a \cong \langle \textbf{x}: \Phi^a(\textbf{x}) \rangle$ is finite for all $a>0$. In terms of subgroups of $F(\textbf{x})$, finiteness of $\hat{G}_a$ for $a>0$ implies the existence of a proper infinite descending chain of normal subgroups in $F(\textbf{x})$, each having finite index in $F(\textbf{x})$ and normally generated by $\{ \Phi^a(x_i): x_i \in \textbf{x} \}$. Notice that such $\Phi$ fail to be surjective in a very weak sense. Of course, there are simple examples with this property, e.g. $\Phi \in \mathrm{End}(F(x_0, x_1))$ by $\Phi(x_0)=x_0$ and $\Phi(x_1)=x_1^2$. However, both of the examples given below have the property that $\Phi(x_i)$ has word length greater than one for all $x_i \in \textbf{x}$. In fact, the word length of $\Phi^a(x_i)$ has exponential growth. 

\begin{Lemma} \label{Lemma: finite iterate} For $\textbf{x}=\{x_0, x_1\}$, let $\Phi$ be the endomorphism of $F(\textbf{x})$ defined by $x_i \mapsto x_i^2x_{i+1}^{-1}$ for $i=0, 1$. Then for $a>0$ the group 
$\hat{G}_a \cong \langle \textbf{x}: \Phi^a(\textbf{x}) \rangle$ is cyclic of order $3^a$ and its normal subgroups $N_k= \langle \Phi^{a-k}(\textbf{x}) \rangle \cong \hat{G}_{k}$ for $0 \leq k \leq a$ form a composition series where $N_k / N_{k-1} \cong \mathbb{Z}_{3}$ for $k>0$.
\end{Lemma}

\begin{proof} The cases $a=1, 2$ are Example 3.1 where $\hat{G}_1 \cong \mathbb{Z}_3$, and $\hat{G}_2 \cong \mathbb{Z}_9$ satisfies $x_0= x_1^{-1}$ and $\hat{G}_2 / \langle  \Phi(\textbf{x}) \rangle \cong \mathbb{Z}_3$. Proceeding inductively, for $a>2$ assume $\hat{G}_{a-1}\cong \mathbb{Z}_{3^{a-1}}$ and satisfies $x_0=x_1^{-1}$. Then $\Phi(x_0)= \Phi(x_1^{-1})$ in $\hat{G}_a$ which implies $x_0^2x_1^{-1}=(x_1^2x_0^{-1})^{-1}$ so that $x_0= x_1^{-1}$ in $\hat{G}_a$, too. Working modulo the relation $x_0=x_1^{-1}$, $\Phi(x_0)= x_0^2x_1^{-1}= x_0^3$ so that the relator $\Phi^a(x_0)= \Phi^{a-1}(x_0^3)=\dots= x_0^{3^a}$. Similarly, the second relator becomes $\Phi^a(x_1) = x_1^{3^a}$. Thus $\Phi^a(x_1)$ is a consequence of $x_0=x_1^{-1}$ and $\Phi^a(x_0)$, from which it follows that $\hat{G}_a \cong \langle x_0: x_0^{3^a} \rangle$.
\qed
\end{proof}

A similar example arises in the free group $F(x_0, x_1, x_2)$ with the substitution $\Phi(x_i)=x_ix_{i+1}$ for $i=0,1,2$. This time $\hat{G}_a$ is cyclic of order $2^a$ for $a>0$, which may also be proved inductively, e.g with the hypothesis $x_0=x_1=x_2$ and $|x_0|=2^{a-1}$ in $G_{a-1}$. We provide an argument that $x_0=x_1$ in case $a=2$, from which the interested reader may work out the general argument. Using that $\hat{G}_1 \cong \langle x_0 \rangle$ is cyclic order 2 where $x_0=x_1=x_2$, in $\hat{G}_2$ we have
\begin{align*}
x_0 &= x_0 \Phi(x_0^{2}) \\
&= x_0 \Phi(x_1x_0) \\
&=x_0x_1x_2x_0x_1 \\
&=\Phi(x_0x_2)x_1 \\
&=\Phi(x_0^2)x_1 \\
&=x_1
\end{align*}
For the general argument, note that if $\Phi(x_0^k)$ ends in $x_0$ for some $k$, then we may instead insert $\Phi (x_0^{2k})$. For since $2k$ is even and $|\textbf{x}|=3$ is odd, the word will necessarily end in $x_1$ or $x_2$ instead. In fact, the same argument applies anytime $F(\textbf{x})$ has odd rank (the argument fails when the rank is even since $\hat{G}_1 \cong \mathbb{Z}$ in that case). In summary we have the following. 

\begin{Lemma} \label{Lemma: another finite iterate} For $\textbf{x}=\{x_0, x_1, \dots, x_{n-1} \}$ with $n$ odd, let $\Phi$ be the endomorphism of $F(\textbf{x})$ defined by $x_i \mapsto x_ix_{i+1}$ for $0\leq i< n$. Then for $a>0$ the group $\hat{G}_a \cong \langle \textbf{x}: \Phi^a(\textbf{x}) \rangle$ is cyclic of order $2^a$ and its normal subgroups $N_k= \langle \Phi^{a-k}(\textbf{x}) \rangle \cong \hat{G}_{k}$ for $0 \leq k \leq a$ form a composition series where $N_k / N_{k-1} \cong \mathbb{Z}_{2}$ for $k>0$.
\end{Lemma}

\section{The substitution homomorphism $\epsilon$} \label{Section: eqn groups}

\subsection{Injectivity} \label{Subsection: injective}

Let us now consider the injectivity problem for the induced map $\epsilon: G \rightarrow \hat{G}$. An important feature of the construction is that $G$ often embeds in $\hat{G}$. This claim may be supported by Theorem \ref{Theorem: relative aspherical}, which provides injectivity when the relative presentation
$$\hat{\mathcal{G}} = \langle G, \mathbf{x}: \epsilon(u_i)u_i^{-1}, \ 0\leq i< k\rangle$$
for $\hat{G}$ from Lemma \ref{Lemma: relative hat G} is aspherical. Furthermore, this  relative presentation is an instance of the adjunction problem from equations over groups, where many classes of relative presentations are known to have an injective natural map. We will explore the latter in detail shortly, and for now begin with some preliminary examples and results.

Perhaps the simplest type of substitution occurs when the generators of $\mathbf{x}$ occurring in each word $\epsilon(u_i)\in F(\textbf{x})$ form a partition of $\textbf{x}$. For here the presentation $\hat{\mathcal{G}}$ shows that $\hat{G}$ may be formed as an iterated free product with amalgamation, and in particular that $G$ embeds in $\hat{G}$. 

\begin{Lemma} Let $\mathcal{G} = \langle \mathbf{u}: \mathbf{v} \rangle$, $\epsilon: F(\mathbf{u}) \rightarrow F(\mathbf{x})$, and $\mathbf{x}_i \subseteq \mathbf{x}$ denote the generators of $F(\mathbf{x})$ appearing in $\epsilon(u_i)$ for $0 \leq i < k$. If $\{\textbf{x}_0, \textbf{x}_1, \dots, \textbf{x}_{k-1} \}$ is a partition of $\textbf{x}$, then $\hat{G}$ may be obtained from $G$ as an iterated free product with amalgamation. Beginning with $G$, at each step $0 \leq i < k$ there is a free product with the one relator group
$$A_i \cong \langle \mathbf{x}_{i}: \epsilon(u_i)^{a_i} \rangle$$
amalgamated along the subgroups generated by $\langle u_i \rangle$ and $\langle \epsilon(u_i) \rangle$, for $a_i$ the order of $u_i$ in $G$ and where $a_i = 0$ if that order is infinite.
\end{Lemma}

\begin{Example} Consider the presentation $\mathcal{P}=\langle x_0, x_1: x_0^2=x_1^3 \rangle$ for the group $\hat{G}$ that is the free product of two copies of $\mathbb{Z}$ amalgamated over the cyclic subgroups generated by $\langle x_0^2 \rangle$ and $\langle x_1^3 \rangle$, respectively. Then $\mathcal{P}$ admits the substitution $\epsilon: F(u_0, u_1) \rightarrow F(x_0, x_1)$ given by $u_0 \mapsto x_0^2$ and $u_1 \mapsto x_1^3$ with $\mathcal{G} = \langle u_0, u_1: u_0= u_1 \rangle$. The substitution is such that $G \cong \mathbb{Z}$ with $\epsilon(G) \cong \langle x_0^2 \rangle \cong \langle x_1^3 \rangle$ and where $\hat{G} \twoheadrightarrow H \cong \mathbb{Z}_2 \ast \mathbb{Z}_3$. Compare with Lemma \ref{Lemma: Q21}, where we arrive at the same group $\hat{G} \cong T$ using a similar presentation for $G \cong \mathbb{Z}$ but with a substitution $\epsilon$ such that $H=1$.

\end{Example}

It is not hard to find examples where the substitution $\epsilon: G \rightarrow \hat{G}$ is not injective, even if we require each member of the substitution set to have reduced length greater than one. We introduce an example now using the non-injective free group endomorphism from Example \ref{Example: empty word composite}. In contrast, we recall a classical result due to Higman \cite{HigHop}, that even if $\Phi_{\epsilon} \in \mathrm{End}(F(\mathbf{x})$ is induced by a Nielsen transformation (and so is injective), then $\Phi_{\epsilon}$ may induce an epimorphism $\Phi_{\epsilon}: G \rightarrow \hat{G} \cong G$ that need not be injective (i.e. $G$ is not Hopfian).

\begin{Example} \label{Example: G not injective} Let $n>1$ and $u=x_0x_1^{-1}$. Then $u \theta_{\mathbf{x}}(u) \dots \theta_{\mathbf{x}}^{n-1}(u)=1 \in F(\textbf{x})$ so that $g=u_0 u_1 \dots u_{n-1}$ is in the kernel of $\epsilon_u: F(\mathbf{u}) \rightarrow F(\mathbf{x})$. Thus $\epsilon_u: G \rightarrow \hat{G}$ is not injective for any group $G$ such that $g \in G$ is nontrivial. A simple choice is to use the same presentation and let $G=G_n(u_0u_1^{-1})\cong \langle u_0 \rangle \cong \mathbb{Z}$ so that $g=u_0^n$ is nontrivial, and therefore $\epsilon_u: G \rightarrow \hat{G}$ is not injective. The composite word is $v \circ u= (x_0x_1^{-1})(x_1x_2^{-1})^{-1}=x_0x_1^{-1}x_2x_1^{-1}$ and the relative presentation for $\hat{G}=G_n(x_0x_1^{-1}x_2x_1^{-1})$ is given by
$$\langle G, x_i: x_ix_{i+1}^{-1}u_i^{-1}, \ 0\leq i< n \rangle .$$
In case $n=2$, we have $G_2(x_0x_1^{-1}x_2x_1^{-1}) \cong G_2((x_0x_1^{-1})^2)$, which may be obtained from $G_2(u_0^2)$ with the same substitution $\epsilon_u$. It is also the case that $g=u_0u_1$ is nontrivial in $G_2(u_0^2)$, so $\epsilon_u$ still has nontrivial kernel. Finally, note that whether the substitution $\epsilon_u$ is applied to $\mathcal{G}_2(u_0u_1^{-1}) \cong \mathbb{Z}$ or $\mathcal{G}_2(u_0^2) \cong \mathbb{Z}_2 \ast \mathbb{Z}_2$, the resulting groups are isomorphic. Thus in general $(\hat{G}, \epsilon)$ does not determine $G$.
\end{Example}

We now recall some definitions from the theory of equations over groups. A system of $k>0$ equations 
$$w_i(x_0, x_1, \dots, x_{n-1})=1$$
in the $n$ variables $\mathbf{x}$ are over a group $G$ if the given $w_i=w_i(x_0, x_1, \dots, x_{n-1})$ are elements of $F(\mathbf{x}) \ast G$ for $0 \leq i<k$, and there is a solution to the system over $G$ if $G$ embeds into a group $A$ that solves the system, i.e. there is an assignment $x_j \mapsto a_j \in A$ for $0 \leq j<n$ such that $w_i \mapsto 1$ for $0 \leq i<k$. It is well known that a system of equations over $G$ corresponds to the relative presentation $\langle G, \mathbf{x}: w_i=1, \ 0\leq i<k \rangle$ of the group $\hat{G}$, where the system has a solution if and only if the natural map from $G$ to $\hat{G}$ is injective. 

Thus for $\mathcal{P}$ a composite presentation given by $(\mathcal{G}, \epsilon)$, the relative presentation $\hat{\mathcal{G}}$ of $\hat{G}$ determines a system of $k$ equations over the group $G$ in the $n$ variables $\mathbf{x}$. Each equation $\epsilon(u_i)u_i^{-1}=1$ has $\epsilon(u_i) \in F(\mathbf{x})$ so that $\epsilon(u_i)u_i^{-1} \in F(\mathbf{x}) \times \mathbf{u}$ for $0 \leq i<k$, i.e. $u_i^{-1} \in G$ is the only nontrivial coefficient. However, if a substitution $\epsilon$ satisfies $u_i \mapsto x_j$ for some $i,j$, i.e. $\mathcal{P}$ is partially composite, then we may reduce the system in order to arrive at equation(s) of the form $\epsilon(u_i)u_i^{-1} \in F(x_0, x_1, \dots , x_{l-1}) \ast G$ for some $l>0$. For instance, Example \ref{Example: partial comp} reduces the $k=2$ equations to the equation $xAxAx^{-1}A^{-2}U^{-1}=1$ in the single variable $x$.

The situation $G\rightarrow \langle G, x: r \rangle $ where $r \in G \ast \{ x \}$ is the most studied in the literature. In our situation, this coincides with a partially composite presentation with one substitution in a single new variable. As in Lemma \ref{Lemma: E relative}, we can demonstrate this situation with the shift extension $E_n(v \circ u)$ for a composite cyclically presented group $G_n(v \circ u)$. In terms of coefficient groups, recall the shift extension $E_n(w)= G_n(w) \rtimes  \mathbb{Z}_n$ always has torsion while the cyclically presented group may not. At the same time, $G_n(u)$ injects into $G_n(v \circ u)$  if and only if $E_n(v)$ injects into $E_n(v \circ u)$ by Theorem \ref{Theorem: injective}.

\begin{Example} \label{Example: not injective E} Returning to Example \ref{Example: G not injective}, the choice $u=x_0x_1^{-1}$ corresponds to the relator $i_{a,x}(u)U^{-1}=xax^{-1}a^{-1}U^{-1}$ and the relative presentation is 
$$\langle E_n(v), x: xax^{-1}a^{-1}U^{-1} \rangle$$
where $E_n(v)=E_n(u_0u_1^{-1}) \cong \langle a, U: a^n, UaU^{-1}a^{-1} \rangle$. This equation is of the well-known form $\langle G, x: xgx^{-1}h \rangle$, which is not solvable when $g$ and $h$ have different orders in $G$; here $E_n(v) \cong \mathbb{Z}_n \times \mathbb{Z}$ where $a$ has order $n$ while $a^{-1}U^{-1}$ has infinite order. Thus while $G_n(u_0u_1^{-1}) \cong \mathbb{Z}$ is torsion free, the substitution $\epsilon: G_n(v) \rightarrow G_n(v \circ u)$ is influenced by torsion in its shift extension. Recall the element $g=u_0u_1 \dots u_{n-1} \neq 1 \in G$ satisfies $\epsilon(g)=1\in \hat{G}$. In terms of the shift extensions, the element is $i_{a,U}(g)=U aUa^{-1} \dots a^{n-1}Ua^{-(n-1)}=U^n\in E_n(v)$ and the substitution $\epsilon(U)=xax^{-1}a^{-1}$ has $i_{a,x}(\epsilon(g))=1$. Notice then that $U$ has infinite order in $E_n(v)$, but $\epsilon(U)$ now has finite order in $E_n(v \circ u)$ since $\epsilon(U^n)=\epsilon(i_{a,U}(g))=i_{a,x}(g)=1$.
\end{Example}

It is well known that equation(s) over a group $G$ will not have a solution in general. One hypothesis that has received considerable attention is when $\langle G, x: r \rangle$ is non-singular, i.e. the exponent sum of $x$ in $w \in G \ast \{x\}$ is nonzero (otherwise the equation is called singular). Similarly, a system of equations is called independent if the $k \times n$ matrix $M$ whose $(i,j)$-th entry is the sum of the exponents of $x_j$ in $w_i$ has linearly independent rows; otherwise, the system is dependent. The system in Example \ref{Example: G not injective} with a composite cyclically presented group is dependent, while the corresponding equation for its shift extension in Example \ref{Example: not injective E} is singular. Dependent systems and singular equations have received far less attention in the literature. See for example \cite{EdjHowSing} and \cite{EdjSingFour} which show large classes of singular length four equations have solutions. 

The previous examples admit many variations where injectivity fails. We give two which will have shift extensions with singular length four equations. Further composing $v \circ u$ with $x=t_0^2$ yields $v \circ u \circ t=t_0^2t_1^{-2}t_2^2t_1^{-2}$. Here there is the relative presentation for the shift extension $E_n(v \circ u \circ t)$ given by $\langle E_n(v), t: t^2at^{-2}a^{-1}U^{-1} \rangle$ which is an instance of an equation having the form $tatbt^{-1}ct^{-1}d=1$. Similarly, the substitution $x=t_0t_1^{-1}$ provides an equation of the form $tat^{-1}btct^{-1}d=1$.

We conclude this section with some applications of results from the theory equations over groups. First we use the result of Rothaus \cite{RothEqn} to recover a structural result for certain one-relator groups from the perspective that a proper power $u^a \in F(\mathbf{x})$ may be regarded as a composite word.

\begin{Example} \label{Example: one relator} Let $a>1$ be a positive integer and $u \in F(\mathbf{x})$ a word having nonzero exponent sum in some $x_i \in \mathbf{x}$. Assume this generator is $x_0$. Then the one- relator presentation $\mathcal{P}=\langle \mathbf{x}: u^a \rangle$ defining the group $\hat{G}$ is composite using $\mathcal{G}= \langle \mathbf{u}: u_0^a \rangle$ with $|\textbf{u}|=n$ and $\epsilon$ defined by $u_0 \mapsto u$ and $u_i \mapsto x_i$ for $0 < i< n$. Now $\hat{\mathcal{G}}$ defines an independent system of $n$ equations in the $n$ variables $\mathbf{x}$, so the result of Rothaus \cite{RothEqn} implies  $\epsilon: G \rightarrow \hat{G}$ is injective. This result applies because the coefficient group $G \cong \mathbb{Z}_a \ast F(u_1, \dots, u_{n-1})$ is residually finite (and so locally residually finite). Thus the subgroup of $\hat{G}$ generated by $\{x_1, x_2, \dots, x_{n-1} \}$ is isomorphic to a free group, and the element $\epsilon(u_0)= u \in \hat{G}$ has order $a>1$. Finally, the substitution $\epsilon$ determines the presentation $\mathcal{H}=\langle \textbf{x}: u, x_i, \ 0<i<n \rangle$, so the factor group $H$ of $\hat{G}$ is finite cyclic with order the absolute value of the exponent sum of $x_0$ in $u$. Thus the normal closure of $\epsilon(G)\cong \mathbb{Z}_a \ast F(n-1)$ in $\hat{G}$ has finite index in $\hat{G}$.
\end{Example}

A similar argument with substitution $u_0 \mapsto x_0x_1 \dots x_{n-1}$ and $u_i \mapsto x_i$ for $0<i<n$ as in the previous example applied to 
$$B(1,2)\ast F(n-2) \cong \langle u_0, u_1, \dots, u_{n-1}:u_1u_0u_1^{-1}u_0^{-2} \rangle$$
yields the presentation from \cite[Corollary 5]{BaumMillRF} for a group $\hat{G}$ which is not residually finite. Thus a composite group $\hat{G}$ need not be residually finite even if $\epsilon$ is injective and $G$ and $H$ are residually finite. Of course, if the given group $G$ is not residually finite and $\epsilon$ is injective, then $\hat{G}$ will not be residually finite.

Let us consider the possible groups which may arise as (a subgroup of) the kernel of some substitution $\epsilon: G \rightarrow \hat{G}$. Of course, any normal subgroup of $F(\mathbf{u})$ maps trivially to $F(\mathbf{x})$ at the level of free groups simply by selecting the substitution which maps its generators to the identity in $F(\mathbf{x})$. So, at the very least, we consider the problem when no generator of $F(\mathbf{u})$ maps trivially and at least one substitution has reduced word length greater than one in $F(\mathbf{x})$.

\begin{Proposition} \label{Prop: subgroup of kernel} Every cyclically presented group $A$ with $n>1$ generators embeds in a cyclically presented group $G_A$ in such a way that it normally generates the kernel of a substitution $\epsilon: G_A \rightarrow \hat{G}_A$. Furthermore, the group $G_A$ is SQ-universal if $n>2$.
\end{Proposition}

\begin{proof} Given $A=G_n(w)$ where $w \in F(\mathbf{y})$ and $\mathbf{y}=\{y_0, y_1, \dots, y_{n-1} \}$, let $y=u_0u_1 \dots u_{n-1}$ and $G_A=G_n(w \circ y)$. Passing to the shift extension $E_n(w \circ y)$, its relative presentation $\langle E_n(w), U: i_{a, Y}(y)U^{-1} \rangle$ with $E_n(w) \cong \langle a, Y: a^n, i_{a, Y}(w) \rangle$ from Lemma \ref{Lemma: E relative} is such that each occurrence of $Y$ in the relator $i_{a, Y}(y)U^{-1}=YaYa \dots Y a^{-(n-1)} U^{-1}=(Ya)^nU^{-1}$ has positive exponent. Thus $E_n(w)$ injects into $E_n(w \circ y)$ by Levin's Theorem \cite{LevEqn}, from which it follows that $A$ injects into $G_A$ by Theorem \ref{Theorem: injective}. Furthermore, $G_A$ is SQ-universal for $n>2$ since it maps onto the substitution group $H=G_n(y)=G_n(u_0 \dots u_{n-1})$, which is free on $\{ u_0, u_1, \dots, u_{n-2} \}$. This is because each relator $\theta_{\mathbf{u}}^i(y)$ for $0<i<n$ in $\mathcal{G}_n(y)$ is a cyclic permutations of $y$, so may be deleted, and $y=1$ is equivalent to $u_{n-1}=(u_0u_1 \dots u_{n-2})^{-1}$. Hence the generator $u_{n-1}$ and its defining relation may be deleted so as to arrive at a presentation for $F(n-1)$. 

Consider now the substitution $\epsilon_u: F (\mathbf{u}) \rightarrow F(\mathbf{x})$ with $u=x_0x_1^{-1}$. Observe $\epsilon_u(y)=1$ (as in Example \ref{Example: empty word composite}) so that $w \circ y \circ u =1$ and $\hat{G}_A \cong G_n(1) \cong F(\mathbf{x})$, from which it follows that the given group $A \cong  \langle \theta^i(y), \ 0\leq i< n \rangle_{G}$ normally generates the kernel of $\epsilon_u$.
\qed
\end{proof}

See \cite{WillCycSQ} for a study of SQ-universality in other classes cyclically presented groups.

We conclude this section with an example showing a substitution may introduce torsion in $\hat{G}$ when the group $G$ is torsion free. 

\begin{Example} For $n>1$, the substitution $\epsilon \in \mathrm{End}(F(\mathbf{x}))$ defined by $x_i \mapsto x_ix_0x_{i+1}^{-1}$ for $0 \leq i<n$ satisfies $g=x_0 x_1 \dots x_{n-1} \mapsto x_0x_0^nx_0^{-1}$. Thus if $g=1$ in a (torsion free) group $G$, then $x_0^n=1$ in $\hat{G}$. Take for example the free group $G \cong F(n-1) \cong \langle \mathbf{x}: x_0x_1 \dots x_{n-1} \rangle$ where $\hat{G}\cong \langle \textbf{x}: x_0^n \rangle$. Then $G$ is torsion free but $x_0$ is nontrivial and has finite order in $\hat{G}$.
\end{Example}




\subsection{Surjectivity} \label{Subsection: surjective}


Surjectivity of the natural map $\epsilon:G \rightarrow \hat{G}$ is rare beyond those presentations already identified in the classical case, where it is an a priori assumption that $\Phi$ is an automorphism of $F$. Recall conditions (\ref{Condition 1}) and (\ref{Condition 2}) given in Section \ref{Section: background} are equivalent to an automorphism $\Phi$ of $F(\mathbf{x})$ inducing an automorphism of a group $G$. Condition (\ref{Condition 2}) requires the substitution be surjective at the level of free groups. In our situation, surjectivity of $\epsilon:F(\mathbf{u}) \rightarrow F(\mathbf{x})$ implies immediately that the induced map $G \rightarrow \hat{G}$ is surjective. However, this is not necessary. 

\begin{Proposition} \label{prop: surjective} For $n, k>0$, let $\mathcal{G}= \langle \mathbf{u}: \mathbf{v} \rangle$ and $\epsilon: F(\mathbf{u}) \rightarrow F(\mathbf{x})$ be a homomorphism. Then the induced map $\epsilon: G \rightarrow \hat{G}$ is surjective if and only if $H \cong \langle \mathbf{x}: \epsilon(\mathbf{u}) \rangle$ is the trivial group and $\epsilon(G)$ is normal in $\hat{G}$.
\end{Proposition}

Said differently, if the substitution group $H \cong \langle \mathbf{x}: \epsilon(\mathbf{v}) \rangle$ is trivial, then Theorem \ref{Theorem: intro short exact} provides that the \textit{normal} closure of $\epsilon(\mathbf{v})$ in $F(\mathbf{x})$ equals all of $F(\mathbf{x})$; however, surjectivity of the induced map concerns only the \textit{subgroup} closure of $\epsilon (\mathbf{u})$ in $\hat{G}$. It is no surprise then that surjectivity of $\epsilon:G \rightarrow \hat{G}$ is algorithmically undecidable. Indeed, given arbitrary $\mathcal{H}$, selecting a compatible presentation $\mathcal{G}$ for the trivial group $G=1$ has the property that $\epsilon$ is surjective if and only if $H=1$ by Theorem \ref{Theorem: intro short exact}, where the triviality problem for a group $H$ is undecidable in general.

It is difficult to construct specific group presentations where the induced map is surjective using nontrivial substitutions. For surjectivity requires working with a substitution $\epsilon$ that determines a presentation $\mathcal{H}$ of the trivial group, as well as a compatible relators $\textbf{v}$ such that $\epsilon(G)$ is normal in $\hat{G}$. Compare with Subsection \ref{Subsection: finiteness} where we discuss how $\epsilon$ rarely preserves the finiteness of $G$, let alone is surjective. Still, we may appeal to Nielsen transformations as in the classical case (e.g. \cite[Section 3.6]{MagKarSol} and \cite{RapFreeAuto}). Recall that such a generator assignment will not always induce a group endomorphism; however, they will always induce an epimorphism $G \rightarrow \hat{G}$. We provide an example as follows.

\begin{Example} \label{Example: surjective}  Let $\mathcal{G}=\langle a, b, c: ab, bc, ca \rangle$ where $G\cong \mathbb{Z}_2$ and $a=b=c$. Define $\Phi_{\epsilon} \in \mathrm{Aut}(F(a,b,c))$ by $a\mapsto a, b \mapsto ab, c \mapsto abc$. Then $G$ maps onto the group $\hat{G}$ given by the composite presentation $\langle a, b, c: a^2b, ababc, abca \rangle$, where $\hat{G} \cong \mathbb{Z}_2$ and $a^2=b=c=1$. Thus $\Phi_{\epsilon}$ is an isomorphism. However, $\Phi_{\epsilon}$ does not induce an automorphism of $G$ in terms of the generators $\{a, b, c \}$ since $a=b=c \in G$, but having $a \mapsto a$ and $b \mapsto ab=1$ is not well defined.
\end{Example}

The surjectivity problem is considered in the literature on equations over groups and related adjunction problems from topology. Here it is assumed a priori that the natural map is injective. We state two results which when applied to our situation provide necessary conditions for a substitution to induce an isomorphism $\epsilon: G \rightarrow \hat{G}$. First, the surjectivity problem in the form $G \rightarrow \langle G, x: W \rangle$ with $W \in G \ast \{ x\}$ is investigated in \cite{CohRourSurj}. They remark that beyond the work of \cite{RothEqn}, it appears to be open for systems of equations. Their main result shows that if $G$ is torsion free, then surjectivity implies $W$ is conjugate to $gx$ or $gx^{-1}$ for some $g\in G$. This theorem applies to those partially composite presentations using one substitution with a single  variable and where the presentation $\mathcal{G}$ is for a torsion free group $G$.

Concerning balanced systems (i.e. $|\textbf{x}|=|\textbf{r}|< \infty$), and for arbitrary groups $G$, another results states that if the natural map $G \rightarrow \langle G, \textbf{x}: \textbf{r} \rangle $ is an isomorphism, then the relative presentation $\langle G, \textbf{x}: \textbf{r} \rangle$ is aspherical \cite[Lemma 2.16]{BogWilEdjAsph}. The requirement that the system is balanced corresponds to those substitutions where $\Phi$ is an endomorphism of $F(\textbf{x})$. Thus in Example \ref{Example: partial comp}, where $\Phi_\epsilon: G \rightarrow \hat{G}$ is an isomorphism, the corresponding relative presentation is not aspherical where $G \cong \langle A, B, C: AB, BC, CA \rangle \cong \mathbb{Z}_2$, $\textbf{x}= \{a, b, c\}$, and $\textbf{r}= \{aA^{-1}, abB^{-1}, abcC^{-1} \}$ . As in \cite[Example 2.17]{BogWilEdjAsph}, this shows that a relative presentation may be aspherical even if an adjoined generator (i.e. $a$) has finite order in $\hat{G}$. 

\section{Substitutions using Higman's Groups} \label{Section: QKL}

In this section we investigate a family of groups $Q(k,l)$ obtained by composing a simple  presentation $\mathcal{G}$ of a cyclic group with a substitution $\epsilon$ whose substitution groups $H$ are modifications of Higman's groups $G_n(x_0x_1x_0^{-1}x_1^{-2})$ \cite{HigInf}. Our overall goal is to examine how the complexity of a composition $(\mathcal{G}, \epsilon )$ may influence the resulting group $\hat{G}$, and in particular to isolate the effect of passing from $\epsilon(G)$ to its normal closure in $\hat{G}$. There is no simpler algebraic scenario than taking $G$ cyclic and $H=1$ in Theorem \ref{Theorem: intro short exact}, and that is the course taken here, for recall Higman's group is trivial when $1 \leq n \leq 3$. Though not our focus, we will have results in the cases when $n>3$ and Higman's groups are infinite, too.

\subsection{Definitions and general results} \label{Subsection: Definitions and general results}

Let $k,l>0$ and define the groups $Q(k,l)$ by 
\begin{align*}
& u =x_0x_lx_0^{-1}x_l^{-2}, \ \mathcal{H} =\mathcal{G}_{kl}(x_0x_lx_0^{-1}x_l^{-2}) \\
& v =u_0u_1, \ \mathcal{G} =\mathcal{G}_{kl}(u_0u_1) \\
& \mathcal{Q}(k, l)=\mathcal{G}_{kl}(v \circ u) =\mathcal{G}_{kl}(x_0x_lx_0^{-1}x_l^{-2}x_1x_{l+1}x_1^{-1}x_{l+1}^{-2})
\end{align*}
where $\mathcal{G}$ and $\mathcal{H}$ present the groups $G$ and $H$, respectively, and the corresponding substitution homomorphism $\epsilon=\epsilon_u: F(\textbf{u}) \rightarrow F(\textbf{x})$ maps $u_i \rightarrow x_ix_{l+i}x_i^{-1}x_{l+i}^{-2}$ for $0\leq i<kl$. The components of the substitution are as follows.

\begin{Lemma} \label{Lemma: length 2} For $n>0$, the group $G_n(u_0u_1)$ is cyclic and generated by $u_0$ where
\begin{equation*}
G_n(u_0u_1) \cong 
    \begin{cases}
        \mathbb{Z} & \text{if } n \ \text{is even},\\
        \mathbb{Z}_2 & \text{if } n \ \text{is odd}.
    \end{cases}
\end{equation*}
Furthermore, the presentations are equivalent via Tietze transformations to
\begin{equation*}
\mathcal{G}_n(u_0u_1) \cong 
    \begin{cases}
         \langle x_0, x_1, \dots, x_{n-1}: u_iu_{i+1}, \ 0 \leq i \leq n-2 \rangle & \text{if } n \ \text{is even},\\
        \langle x_0, x_1, \dots, x_{n-1}: u_0^2, u_iu_{i+1}, \ 0 \leq i \leq n-2 \rangle & \text{if } n \ \text{is odd}.
    \end{cases}
\end{equation*}
\end{Lemma} 

\begin{proof} The relations $u_iu_{i+1}=1$ imply $\theta_{\mathbf{u}} (u_i)=u_{i+1}=u_i^{-1}$ for $0\leq i <n$ so that $\theta_{\mathbf{u}}=-\mathrm{Id}$. If $n$ is odd, this implies $\theta_{\mathbf{u}}=1\in \mathrm{Aut}(H)$ is the identity. Here $G_n(u_0u_1) \cong  G_1(u_0^2) \cong \mathbb{Z}_2$. If $n$ is even, then $\theta^2_{\mathbf{u}}=1 \in \mathrm{Aut}(H)$ and $G_n(u_0u_1)\cong G_2(u_0u_1)$. Then $u_1u_0$ is a consequence of $u_0u_1$ and $G_n(u_0u_1) \cong \langle u_0, u_1: u_0=u_1^{-1} \rangle \cong \mathbb{Z}$. Now the relator $u_{n-1}u_0$ is a consequence of the other $n-1$ relators when $n$ is even. When $n$ is odd, the relator $u_{n-1}u_0$ may be transformed to $u_0^2= (u_0u_1)(u_1^{-1}u_2^{-1}) \dots (u_{n-3}u_{n-2})(u_{n-2}^{-1}u_{n-1}^{-1})(u_{n-1}u_0)$.
\qed
\end{proof}

The substitution groups $H$ are a modification of Higman's groups $G_n(x_0x_1x_0^{-1}x_1^{-2})$ using a standard isomorphism from the theory of cyclically presented groups (e.g. \cite[Section 2]{EdjHamThoTriv}). This introduces the second parameter $l$ and makes for a richer family of groups. Specifically, this will allow for the case of Theorem \ref{Theorem: knot groups} having $k=3$ and which utilizes Higman's presentation of the trivial group with $n=3$ generators.

\begin{Lemma} \label{Lemma: free product higman} For $k,l>0$, let $u=x_0x_lx_0^{-1}x_l^{-2}$. Then 
$$H=G_{kl}(u) \cong \underset{i=1}{\overset{l}{ \ast }} G_k(x_0x_1x_0^{-1}x_1^{-2})$$
is a free product of $l$ copies of Higman's group $G_k(x_0x_1x_0^{-1}x_1^{-2})$. In particular, $H$ is trivial if $1\leq k \leq 3$ and infinite otherwise (in fact SQ-universal).
\end{Lemma}

\begin{proof} The generators and relations of $\mathcal{G}_{kl}(u)$ break up into a disjoint union of $l$ subpresentations of the form $\langle x_{il+j}: \theta^{il}(x_{j}x_{j+l}x_{j}^{-1}x_{j+l}^{-2}), \ 0 \leq i<k \rangle$ where $0\leq j< l$. For each of these subpresentations, an isomorphism with Higman's group $G_k(x_0x_1x_1^{-1}x_1^{-2})$ is given by $x_{il+j} \mapsto x_i$ for $0\leq i<k$. Triviality of Higman's groups for $1 \leq k \leq 3$ is in \cite{HigInf} while for $k=4$ Schupp proved the groups are SQ-universal \cite{SchuSQ}. An easy modification of his proof using blocking pairs shows that Higman's groups are in fact SQ-universal for all $k>3$.
\qed
\end{proof}

\begin{Lemma} \label{Lemma: Knot Groups ab sc}  For $k,l>0$ and $Q=Q(k,l)$, the abelianization of $Q$ satisfies $Q^{\mathrm{ab}} \cong G_{kl}(v)$ where

\begin{equation*}
Q^{\mathrm{ab}} \cong
    \begin{cases}
        \mathbb{Z} & \text{if } kl \ \text{is even},\\
        \mathbb{Z}_2 & \text{if } kl \ \text{is odd}.
    \end{cases}
\end{equation*}
Consequently, $\epsilon:G_{kl}(u) \rightarrow Q(k,l)$ is injective for all $k,l>0$.
\end{Lemma} 

\begin{proof} Under abelianization $u=x_0x_lx_0^{-1}x_l^{-2} \overset{ab}{\mapsto} x_l^{-1}$ so that $v \circ u \overset{ab}{\mapsto} x_l^{-1}x_{l+1}^{-1}$. But then inverting and cyclically permuting each relation $x_l^{-1}x_{l+1}^{-1}=1$ provides the relations for the abelian group $G_{kl}(x_0x_1)$. Thus $Q^{ab} \cong G_{kl}(x_0x_1)$ and it follows that $\epsilon$ is injective. 
\qed
\end{proof}

Recall the class of groups $\mathcal{G}(\mathcal{B})$ are defined in Section \ref{Subsection: normal closure} and are $2$-knot groups by Theorem \ref{Theorem: Z by 1}. Applying Theorem \ref{Theorem: intro short exact} now provides the following.

\begin{Theorem} \label{Theorem: Q structure} \label{Theorem: knot groups} Let $l,k>0$ and $Q=Q(k,l)$. If $1\leq k \leq 3$, then $H=1$ and $Q$ is normally generated by $u=x_0x_lx_0^{-1}x_l^{-2}$. If additionally $kl$ is even then $Q \in \mathcal{G}(\mathcal{B})$, while if $kl$ is odd then $Q$ is finitely generated by involutions and is the quotient of a group in $\mathcal{G}(\mathcal{B})$ by the normal closure of $u^2$.  If $k>3$, then $Q$ is SQ-universal.
\end{Theorem}

\begin{proof} For $k>3$, Corollary \ref{Corollary: quotient} [ii] provides that the groups are SQ-universal. Assume $1\leq k \leq 3$. The groups $Q(k,l)$ are defined in terms of a substitution on $\mathcal{G}_{kl}(u_0u_1)$. When $kl$ is even, Lemma \ref{Lemma: length 2} shows the last relator is redundant and may be deleted. Hence in this case $Q(k,l)$ is isomorphic to a composite group defined by a substitution on a deficiency one presentation for $\mathbb{Z}$, with substitution group $H=1$, so $Q \in \mathcal{G}(\mathcal{B})$. If $kl$ is odd, then Lemma \ref{Lemma: length 2} shows $\mathcal{G}_{kl}(u_0u_1)$ has $u_0^2$ adjoined to a deficiency one presentation for $\mathbb{Z}$. Thus $Q(k,l)$ for $kl$ odd is the quotient of a group in $\mathcal{G}(\mathcal{B})$ modulo the normal closure of $\epsilon(u_0^2)=u^2$. Because $u$ normally generates $Q(k,l)$, which itself is finitely generated, we conclude $Q(k,l)$ is finitely generated by conjugates of $u$.
\qed
\end{proof}

We have the following result concerning finite quotients of $Q(k,1)$ and a more general result for arbitrary $l>0$. Thus $Q(k,l)$ has nontrivial finite quotients even though Higman's groups $G_k(u)$ do not when $k>3$. Also, except for $Q(1,1)\cong \mathbb{Z}_2$, where $u$ collapses to $u=x_0^{-1}$, these results imply $\epsilon(G)\cong G$ is not normal in $Q(k,l)$, i.e. $Q(k,l) \ncong G$ unless $k=l=1$. 

\begin{Proposition} \label{Proposition: factor dihedral} Let $k>0$. Then $Q=Q(k,1)$ modulo $M$, the normal subgroup generated by $\{x_i^2, \ 0 \leq i < k \}$, is isomorphic to $D_m$ the dihedral group of order $2m$ for $m=2^k-1$.
\end{Proposition}

\begin{proof} Let $u'=x_0x_1x_0^{-1}$ with $G=G_{n}(v)$. Then  $G_{k}(\epsilon_{u'}(v)) \cong D_m$ for odd $k$ where $m=2^k-1$, and where $x_i^2=1$ for $0 \leq i< k$, by Example \ref{Example: Dihedral}. In fact, if we adjoin the relations $x_i^2=1$ for $0 \leq i< k$ to $\mathcal{G}_{k}(\epsilon_{u'}(v))$ when $k$ is even, then the same proof shows $\mathcal{G}_{k}(\epsilon_{u'}(v))$ modulo the normal closure of $\{x_i^2, \ 0\leq i<k \}$ is isomorphic to $D_m$, too. Thus the dihedral group $D_m$ may also be presented by 
\begin{align*}
D_m & \cong \langle \mathbf{x}: x_ix_{i+1}x_i^{-1} x_{i+1}x_{i+2}x_{i+1}^{-1}, \ x_i^2, \ 0 \leq i< k \rangle \\
& \cong \langle \mathbf{x}: (x_ix_{i+1}x_i^{-1})x_{i+1}^{-2}(x_{i+1}x_{i+2}x_{i+1}^{-1})x_{i+2}^{-2}, \ x_i^2, \ 0 \leq i< k \rangle \\
&= \langle \mathbf{x}: \theta^i(\epsilon_u(u_0u_1)), \ x_i^2, \ 0 \leq i< k \rangle 
\end{align*}
which is a presentation for $Q/M$. 
\qed
\end{proof}

When $l>1$, we may rearrange $Q(k,l)/M$ into disjoint subpresentations with amalgamations to obtain the following.

\begin{Theorem} \label{Theorem: Q mod M} The group $Q=Q(k,l)$ modulo the normal subgroup generated by $\{x_i^2, \ 0 \leq i < k \}$ is isomorphic to a free product of $l$ copies of the dihedral group $D_m$ for $m=2^k-1$, with each factor generated by $\{ x_{a+jl}, \ 0 \leq j < k \}$ for $0 \leq a < l$, amalgamated along the subgroups generated by the involutions $\theta^i (u)$ for $0 \leq i < l$.
\end{Theorem}

\begin{proof} Let $M$ be the normal closure of $ \{x_i^2, \ 0 \leq i < k \}$ in $Q=Q(k,l)$ for $k,l>0$. Proposition \ref{Proposition: factor dihedral} proves $Q / M$ is isomorphic to $D_m$ for $m=2^k-1$ when $l=1$. For arbitrary $l>0$, the relations $x_i^2=1$ and the shift dynamics of $G$ imply $\theta^i(u)= \theta^j(u)$ for $0 \leq i,j< kl$. Thus for $u''=x_0x_lx_0^{-1}$, the factor group $Q/ M$ satisfies
\begin{align*}
Q \big/ M & \cong \langle \mathbf{x}: \theta^i(u \theta(u)), \theta^i(u)= \theta^j(u), x_i^2, \ 0 \leq i, j < lk \rangle \\
& \cong \langle \mathbf{x}: \theta^i(u'' \theta(u'')), \theta^i(u'')= \theta^j(u''), x_i^2, \ 0 \leq i, j < lk \rangle \\
& \cong \langle \mathbf{x}: \theta^i(u'' \theta^l(u'')), \theta^i(u'')= \theta^j(u''), x_i^2, \ 0 \leq i, j < lk \rangle .
\end{align*}
This presentation contains $l$ disjoint subpresentations generated by $\{ x_{a+jl}, \ 0 \leq j < k \} $ for $0 \leq a < l$ and having relations $\{ \theta^{a+jl}(u'' \theta^l(u'')), \ 0 \leq j< k \}$. Each subpresentation is isomorphic to $G_k(\epsilon_{u'}(v))$ of Example \ref{Example: Dihedral}, and therefore presents $D_m$ for $m=2^k-1$. In each factor isomorphic to $D_m$ the relations $\theta^a(u'')= \theta^{a+jl}(u'')$ for $0 \leq j < k$ are already satisfied, so we need only retain $u''=\theta^i (u'')$ for $0 \leq i < l$ in our presentation for $Q/ M$. Consequently, $Q/ M$ is isomorphic to a free product of $l$ copies of $D_m$, amalgamated along the subgroups of order 2 determined by the relations $u''=\theta^i (u'')$ for $0 \leq i < l$. \qed
\end{proof}

As discussed in Section \ref{Subsection: normal closure}, the deficiency condition for a 2-knot group is stronger than the homology condition for an $n$-knot group with $n\geq 3$. In fact, we will prove that the second homology vanishes for all groups $Q(k,l)$.

\begin{Lemma} For $k,l>0$ and $Q=Q(k,l)$, the second homology group $H_2(Q)$ is trivial.
\end{Lemma}

\begin{proof} Let $K$ be the cellular model of $\mathcal{P}_{kl}(v \circ u)$. If $kl$ is odd, then $H_2(K)=0$ since $H_2(K)$ is free abelian and $H_1(K)\cong G_{kl}(v)^{ab} \cong \mathbb{Z}_2$ by Lemmas \ref{Lemma: length 2} and \ref{Lemma: Knot Groups ab sc}. In particular, the Euler characteristic of $K$ is $\chi(K)=1-kl+kl=1$ where in homology
\begin{align*} 
\chi(K) & =\mathrm{rank}(H_0(K))- \mathrm{rank}(H_1(K)) + \mathrm{rank}(H_2(K)) \\
& = 1- 0 + \mathrm{rank}(H_2(K))
\end{align*}
so that $ \mathrm{rank}(H_2(K))=0$. 

Now suppose $kl$ is even. Here $H_1(K) \cong G_{kl}(v)^{ab}\cong \mathbb{Z}$ so that $1=1-1+\mathrm{rank}(H_2(K))$ implies $H_2(K)\cong \mathbb{Z}$. Indeed, for $0\leq i< kl$, if $c^2_i$ is the $2$-cell with attaching map $\theta_{\mathbf{x}}^i(v \circ u)$ then $\partial c^2_i = -c^1_{i+l}-c^1_{i+l+1}$ and 
\begin{align*}
H_2(K) \cong \mathrm{Ker}(\partial)  & = \left\langle \sum_{i=0}^{kl-1} (-1)^ic^2_i \right\rangle \\
& = \langle c^2_0 - c^2_1 + \dots - c^2_{kl-1} \rangle .
\end{align*}
Now the spherical diagram in Figure \ref{Comp_Length_Two} over $\mathcal{P}_n(v)$ passes (via $\epsilon_u$) to a spherical diagram $\Sigma$ over $\mathcal{P}_n(v \circ u)$. The homology class of $\Sigma$ is $h([\Sigma])=  \sum_{i=0}^{kl-1} (-1)^ic^2_i$, for $h: \pi_2(K) \rightarrow H_2(K)$ the Hurewicz homomorphism (\cite{Sier2dh}, p. 76). Thus $\pi_2(K) \twoheadrightarrow H_2(K)$ is surjective. Finally, Hopf's Theorem (\cite{BrownCoh}, pp. 41-42) provides that  
$$\pi_2(K) \rightarrow H_2(K) \rightarrow H_2(\pi_1(K)) \rightarrow 1$$
is exact so whether $kl$ is even or odd we have that $H_2(\pi_1(K)) =H_2(Q(k,l)) =0$. \qed
\end{proof}

\begin{figure} 
\centering
\includegraphics[width=.4\linewidth]{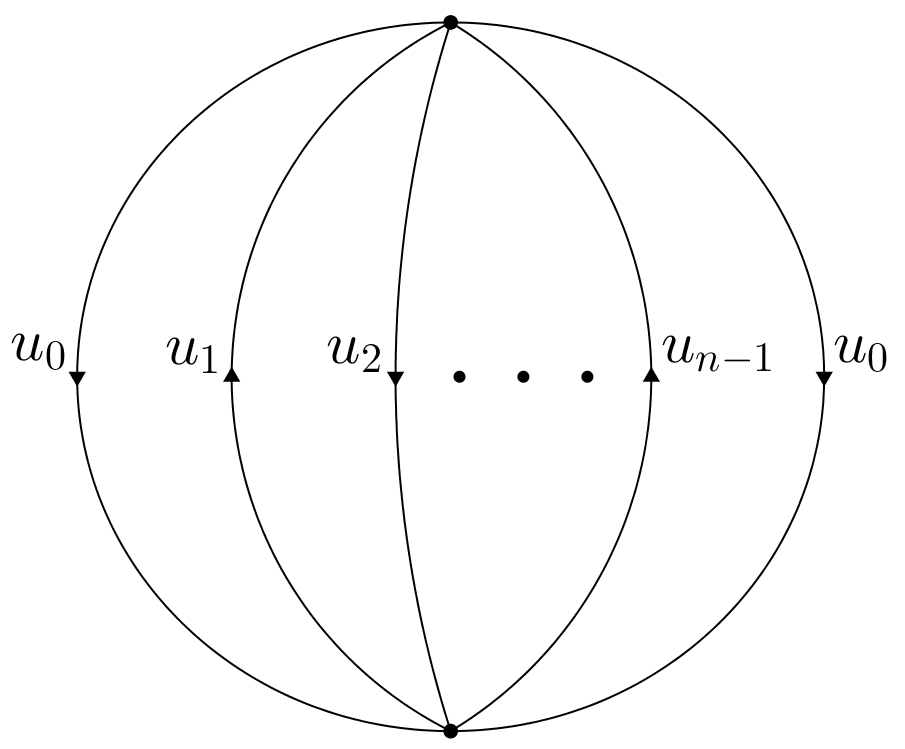}
\caption[Spherical diagram over $\langle u_i: u_i u_{i+1} \ 0 \leq i < n\rangle $ ]{Portion of a spherical diagram over $\langle u_i: u_i u_{i+1}, \ 0 \leq i < n\rangle $.}
\label{Comp_Length_Two}
\end{figure}

\subsection{$Q(k,l)$ for small $k,l$} \label{Subsection: Q small k,l}

We now investigate the groups $Q(k,l)$ for small values of $k$ and $l$. It will become clear that these groups serve as exemplars for a  decomposition of the groups $Q(k,l)$ with arbitrary $k$ and $l$. \\

\noindent \textbf{Case} $\textbf{k=1}: \ $ For $k=1$ the substitution collapses to $u=x_0^{-1}$, so $Q(1,l) \cong \mathbb{Z}$ or $\mathbb{Z}_2$ depending on if $l$ is even or odd, respectively. \\

\noindent\textbf{Case}  $\textbf{k=2}: \ $ The group $Q(2,1)$ admits the well known presentation $\langle x, y: xyx= yxy \rangle$ for the group $T$ of the trefoil knot. See for example \cite[Chapter 3]{RolfKnot} where it is also shown that $T$ admits the presentation $\langle a, b: a^2=b^3 \rangle$.

\begin{Lemma} \label{Lemma: Q21} The group $Q(2,1)$ is isomorphic to the trefoil group $T$.
\end{Lemma}

\begin{proof} Here $u=x_0x_1x_0^{-1}x_1^{-2}$ and the first relator 
$$\epsilon(u_0u_1)=x_0 x_1x_0^{-1}x_1^{-2}x_1x_0x_1^{-1}x_0^{-2}$$
is conjugate to $x_0^{-1}x_1x_0^{-1}x_1^{-1}x_0x_1^{-1}$. In fact, $\epsilon(u_0u_1)$ is also conjugate to the second relator $\epsilon(u_1u_0)$. Thus $Q(2,1) \cong \langle x_0, x_1: x_0^{-1}x_1x_0^{-1}x_1^{-1}x_0x_1^{-1} \rangle$. The change of variables $\bar{x}_0= x_0^{-1}$ and writing the relator as the relation $\bar{x}_0x_1 \bar{x}_0= x_1 \bar{x}_0x_1$ yields the desired presentation.
\qed
\end{proof}

Consider now the group $Q(2,2)$. We show it is a free product with amalgamation in terms of the group given by 
$$\overline{T}= \langle a, b: aba^{-1}b^{-2}=bab^{-1}a^{-2} \rangle$$
which recall was given in Example \ref{Example: T bar not K1} as an example of a group in $\mathcal{G}(\mathcal{B})$ which is a 2-knot group but not a 1-knot group.

\begin{Lemma} \label{Lemma: Q(2,2)} The group $Q(2,2)$ is isomorphic to a free product of two copies of $\bar{T}$ amalgamated along the infinite cyclic subgroups generated by $u=aba^{-1}b^{-2}$ and $\theta(u)=bab^{-1}a^{-2}$, respectively.  That is, $Q(2,2) \cong \bar{T} \underset{\mathbb{Z}}{\ast} \bar{T}$.
\end{Lemma}

\begin{proof} Delete the superfluous relator $u_4u_0$ from $\mathcal{G}_4(u_0u_1)$, and then there is the equivalence
$$\{ u_0u_1, u_1u_2, u_2u_3 \} \leftrightarrow \{u_0u_2^{-1}, u_1u_3^{-1}, u_2u_3 \} \leftrightarrow \{u_0u_2^{-1}, u_1u_3^{-1}, u_0u_1 \}$$
using Nielsen transformations. In terms of $\mathcal{Q}(2,2)$, the superfluous relator is $\epsilon(u_4u_0)$. Notice the relator $\epsilon(u_0u_2^{-1})$ is a word in $F(x_0, x_2)$; $\epsilon(u_1u_3^{-1})$ is in $F(x_1, x_3)$. Using that $\bar{T} \cong \langle x_i, x_{i+2}: \epsilon(u_{i}u_{i+2}^{-1}) \rangle$ for $i=0,1$, the group $Q(2,2)$ is isomorphic to a free product of two copies of $\bar{T}$ with amalgamation determined by the third (and final) relation $u=\theta(u)^{-1}$. The group $\bar{T}$ is a torsion free one-relator group, so the amalgamation is along an infinite cyclic group.
\qed
\end{proof}

Notice the groups $T$ and $\bar{T}$ both admit a 2-generator, one relator presentation of the form $u\theta(u)^{\pm 1}$. In fact, they also share a common finite quotient. Incidentally, this provides another example of a finite group given by a composite presentation.

\begin{Lemma} \label{Lemma: GL23}  The quotient of the groups $Q(2,1) \cong T$ and $\bar{T}$ modulo the normal subgroup $\langle \langle u^2 \rangle \rangle$ are isomorphic to  $GL(2,3)$. Modulo $\langle \langle x_1^2 \rangle \rangle$, each of the groups is isomorphic to the symmetric group $S_3$. 
\end{Lemma}

\begin{proof} The relation $u^2=1$ implies $\theta(u)^2=1$ so that $u\theta(u)=u\theta(u)^{-1}$ in each of the factor groups. Hence they are isomorphic. That the quotient is finite and isomorphic to $GL(2,3)$ can be verified in GAP. Now $u^2=1$ is a consequence $x_1^2=1$. That the quotient modulo $\langle \langle x_1^2 \rangle \rangle$ is $S_3$ can also be verified using GAP.
\qed
\end{proof}

Notice that $Q(2,1)$ modulo $\langle \langle x_0^2, x_1^2 \rangle \rangle$ is isomorphic to $D_3 \cong S_3$ by Proposition \ref{Proposition: factor dihedral}, too.

As a final comparison of $T$ and $\bar{T}$, we recall that knots have an invariant known as their finite cyclic branched covers, whose fundamental group admits a cyclic presentation. Sieradski's cyclically presented groups $S(2,n)= G_n(x_0x_2x_1^{-1})$ \cite{SierSqua} are the fundamental group of the $n$-fold cyclic covering of $S^3$ branched over the trefoil knot \cite{CavHegKimSier}, where the first $n=5$ groups are finite and are $1, \mathbb{Z}_3, Q_8, SL(2,3),$ and $SL(2,5)$ \cite{JohnMawFibType}. While $\bar{T}$ is not a 1-knot group, we may perform the same process to obtain 
\begin{align*}
\bar{T} & = \langle a, b: bab^{-1}a^{-2}b^2ab^{-1}a^{-1} \rangle \\
&=\langle a, b, x: bab^{-1}a^{-2}b^2ab^{-1}a^{-1}, x=ba^{-1} \rangle \\
&=\langle a, x: (xa)a(a^{-1}x^{-1})a^{-2}(xa)^2a(a^{-1}x^{-1})a^{-1}\rangle \\
&=\langle a, x: xax^{-1}a^{-2}(xa)^2x^{-1}a^{-1}\rangle \\
&=\langle a, x: x(ax^{-1}a^{-1})(a^{-1}xa)x(ax^{-1}a^{-1})\rangle
\end{align*}
where $\bar{T}/ \langle \langle a^n \rangle \rangle$ is a shift extension (as in Section \ref{Subsection: cyc pres}) and has cyclically presented retraction kernel given by $G_n((x_0x_1^{-1})^2x_{n-1})$. Here the first $n=4$ groups are finite which can be proved using GAP.

\begin{Lemma} The groups $G(n)=G_n((x_0x_1^{-1})^2x_{n-1})$ have $G(1) \cong 1, G(2) \cong \mathbb{Z}_3, G(3) \cong \mathbb{Z}_{13}$, and $G(4)$ a solvable group of order $39,000$. The group $G(5)$ is infinite with abelianization $\mathbb{Z}_{11}^2$.
\end{Lemma} 

Next consider $Q(2,3)$. Here there are six generators and as in the decomposition of $Q(2,2)$ we find an isomorphism with three copies of $T$ amalgamated over $u=\theta(u)^{-1}=\theta^2(u)$. In this way, for $k=2$ the pattern repeats as $l$ increases, and the structure of $Q(2,l)$ alternates between an increasing number of copies of $T$ and $\bar{T}$. \\

\noindent \textbf{Case}  $\textbf{k=3}: \ $ The groups $Q(3,l)$ for $l>0$ utilize Higman's presentation of the trivial group on $n=3$ generators. These groups may be broken down into analogous groups
\begin{align*}
T_3 &\cong \langle a, b, c: u \theta(u), \theta(u)\theta^2(u) \rangle \\
\bar{T}_3 &\cong \langle a, b, c: u = \theta(u), \theta(u)= \theta^2(u) \rangle
\end{align*}
as in Table \ref{Table: Q stuff}. The groups $Q(3,2k)$ for $k>0$ are in $\mathcal{G}(\mathcal{B})$ by Theorem \ref{Theorem: Q structure} and are built from $2k$ copies of the group $\bar{T}_3$. When $l$ is odd, the groups $Q(3,2k+1)$ for $k\geq 0$ are quotients of groups in $\mathcal{G}(\mathcal{B})$ modulo the normal closure of $u^2$. These groups are generated by involutions. The group $Q(3,1)$ will be investigated shortly where $Q(3,1)\cong T_3 / \langle \langle u^2 \rangle \rangle$. The group $Q(3,3)$ was shown to be infinite using computational methods in Example \ref{Example: normally generated by involutions}. For brevity, $Q(3,3)$ is described as $T_3 \underset{\mathbb{Z}_2}{\ast} T_3 \underset{\mathbb{Z}_2}{\ast} T_3 / \langle \langle u^2 \rangle \rangle$ in Table \ref{Table: Q stuff}, but more precisely the quotient modulo $\langle \langle u^2 \rangle \rangle$ is taken in each copy of $T_3$ prior to amalgamation. Equivalently then, $Q(3,3)$ is isomorphic to the free product of three copies of $Q(3,1)$ amalgamated over the subgroups $\langle u \rangle$. This pattern continues in that the groups $Q(3,2k+1)$ are built from $2k+1$ copies of $Q(3,1)$.

\begin{table}[t] 
\begin{center}
\begin{tabular}{|c|c|c|c|} \hline
$k \backslash l$  & 1 & 2 & 3\\ \hline
1 & $\mathbb{Z}_2$ & $\mathbb{Z}$ & $\mathbb{Z}_2$ \\ \hline
2 & \begin{tabular}{c} $T$ \end{tabular} & \begin{tabular}{c} 2-knot group, \\ $\bar{T} \underset{\mathbb{Z}}{\ast} \bar{T}$ \end{tabular}  & \begin{tabular}{c} 2-knot group, \\ $T \underset{\mathbb{Z}}{\ast} T \underset{\mathbb{Z}}{\ast} T$ \end{tabular} \\ \hline
3 & \begin{tabular}{c} gen. by invol., \\ $T_3 / \langle \langle u^2 \rangle \rangle$ \end{tabular} & \begin{tabular}{c} 2-knot group, \\ $\bar{T}_3 \underset{\mathbb{Z}}{\ast} \bar{T}_3$ \end{tabular} & \begin{tabular}{c} gen. by invol., \\ $T_3 \underset{\mathbb{Z}_2}{\ast} T_3 \underset{\mathbb{Z}_2}{\ast} T_3 / \langle \langle u^2 \rangle \rangle$ \end{tabular} \\  \hline
\end{tabular}
\end{center}
\caption{The groups $Q(k,l)=G_{kl}(\epsilon (u_0u_1)))=G_{kl}(x_0x_lx_0^{-1}x_l^{-2}x_1x_{l+1}x_1^{-1}x_{l+1}^{-2})$ for $1\leq k, l \leq 3$.}
\label{Table: Q stuff}
\end{table}

\begin{Lemma} The groups $T_3$ and $\bar{T}_3$ in $\mathcal{G}(\mathcal{B})$ are 2-knot groups but not 1-knot groups.
\end{Lemma}

\begin{proof} The groups are in $\mathcal{G}(\mathcal{B})$ by construction and the Alexander polynomials for $T_3$ and $\bar{T}_3$ are $2t^3-2t^2+2t-1$ and $3t^4-6t^3+3t+1$, respectively. The proof is routine and omitted but note that the polynomials may be verified using the HAP package \cite{HAP} in GAP. The polynomials are not symmetric so these are not the groups of 1-knots.
\qed
\end{proof}

We have $Q(3,1)\cong T_3 / \langle \langle u^2 \rangle \rangle$, where as with $T$ and $\bar{T}$, $T_3 / \langle \langle u^2 \rangle \rangle$ is isomorphic to $\bar{T}_3 / \langle \langle u^2 \rangle \rangle$. However, this time it could not determined if the quotient is finite. That is, the group $Q(3,1)$ may be finite. Some evidence for finiteness comes from an analogous quotient: the group presented by 
$\langle a,b,c: u'\theta(u'), \theta(u')\theta^2(u') \rangle$ with $u'=x_0x_1x_0^{-1}$ is in $\mathcal{G}(\mathcal{B})$ with polynomial $t^3-3t^2+2t-1$, so is not a 1-knot group, and modulo the normal closure of $(u')^2$ has quotient $D_7$. What makes this situation different from $Q(3,1)$ is that the relation $(u')^2=1$ is equivalent via cyclic reduction to the stronger relation $x_1^2=1$.

Recall the quotient of $Q(3,1)$ modulo $M= \langle \langle x_0^2, x_1^2, x_2^2 \rangle \rangle$ is isomorphic to $D_7$ by Proposition \ref{Proposition: factor dihedral}. In fact, modulo any two of $\{ x_0^2, x_1^2, x_2^2 \}$ the factor group is isomorphic to $D_7$. Consider then the quotient
$$B= Q(3,1)/ \langle \langle x_2^2 \rangle \rangle $$
by a square of one of the generators. Equivalently, $B$ is the quotient of $T_3$ or $\bar{T}_3$ by $\langle \langle x_2^2 \rangle \rangle$ since $\theta(u)^2$ is a consequence of $x_2^2$. We show $B$ is a quotient of the group $A \cong \langle x_0, x_1: u^2 \rangle$ from \cite[p. 72]{NewDis}, which B. B. Newman shows is not residually torsion free nilpotent. Indeed, he shows the element $x_1^2 \in A$ is nontrivial and contained in all terms of the lower central series. Notice that we kill this element when forming $Q(3,1)/M \cong D_7$, but nevertheless the quotient is not nilpotent (the nilpotent dihedral groups have order a power of two).

\begin{Lemma} The group $B= Q(3,1)/ \langle \langle x_2^2 \rangle \rangle$ is the quotient of the group $A \cong \langle x_0, x_1: u^2 \rangle$ by the relation $u=(x_1^{-1}u^{-1}x_1)x_0(x_1^{-1}u^{-1}x_1)^{-1}x_0^{-2}$.
\end{Lemma}

\begin{proof} Beginning with the defining presentation $\mathcal{Q}(3,1)$, multiply the third relator by inverses and cyclic permutations of the first two to get $\theta(u)^{-2}$, equivalently $\theta(u)^2$. Then adjoin $x_2^2$ to present the quotient $B$, and notice $\theta(u)^2=(x_1x_2x_1^{-1}x_2^{-2})^2$ is a consequence of $x_2^2$. Thus
\begin{align*}
B & \cong \langle x_0, x_1, x_2: u\theta(u), \theta(u)\theta^2(u), \theta(u)^2, x_2^2 \rangle \\
&\cong \langle x_0, x_1, x_2: u\theta(u), \theta(u)\theta^2(u), x_2^2 \rangle.
\end{align*}
Now $u\theta(u)=ux_1x_2x_1^{-1} \in B$ so $x_2=x_1^{-1}u^{-1}x_1$. It follows that $u^2=1$ is equivalent to $x_2^2=1$, so we may replace $x_2^2$ with $u^2$. Furthermore, we may substitute $x_2=x_1^{-1}u^{-1}x_1$ into the relator $\theta(u)\theta^2(u)$ and delete $x_2$ and its defining relation to get
$$B \cong \langle x_0, x_1: u^2, u^{-1}(x_1^{-1}u^{-1}x_1)x_0(x_1^{-1}u^{-1}x_1)^{-1}x_0^{-2} \rangle $$
as desired.
\qed
\end{proof}

For an alternative description of the quotient $B$, begin with $C \cong \langle x_0, x_2: x_2^2, \theta^2(u)^2 \rangle$. Note this is isomorphic to $A/ \langle \langle x_0^2\rangle \rangle $ but we have switched to the generators $\{x_0, x_2 \}$. It is not hard to show that $C$ is a shift extension of the form  $C \cong T \rtimes \mathbb{Z}_2$ for $T$ the trefoil group. Then form the HNN-extension
$$D= \langle C, x_1: x_1x_2x_1^{-1}= \theta^2(u)^{-1} \rangle$$
whose defining relation is equivalent to $u\theta^2(u)=1$ modulo $x_2^2=1$. Finally, $B$ is isomorphic to $D$ modulo the relation $x_2=x_1^{-1}ux_1=x_1^{-1}x_0x_1x_0^{-1}x_1^{-1}$.

To conclude we note that the group $A$ can be used in a similar way to describe $Q(3,1)$ itself. Starting with $A$ form the HNN-extension $J=\langle A, x_2: x_2x_0x_2^{-1}=ux_0^2 \rangle$. The elements $x_0$ and $ux_0^2$ are not conjugate to $u$ in $A$ so generate infinite cyclic subgroups (for $A$ is a one relator group where all torsion elements are conjugate to the root $u$). Now $Q(3,1)$ is the quotient of $J$ by $u\theta(u)=1$.

\bibliographystyle{plain}
\bibliography{mybib-copy}

\end{document}